\documentclass[12pt, twoside]{article}
\usepackage{dsfont}
\usepackage{amsmath,amsthm,amssymb}
\usepackage{times}
\usepackage{enumerate}
 \usepackage[colorlinks=true]{hyperref}
\usepackage[noabbrev]{cleveref}
\usepackage[many]{tcolorbox}
\usepackage{nicefrac}
\usepackage{stmaryrd}
\usepackage{tikz} 
\usepackage{pgfplots}
\pgfplotsset{compat=newest}
\usetikzlibrary{arrows.meta} 
\usepackage{tkz-euclide}
\crefname{in}{inequality}{inequalities}
\creflabelformat{in}{#2{\upshape(#1)}#3} 
\crefname{assumption}{Assumption}{assumptions}

\newcommand{\di}[1]{\,\mathrm{d}#1}
\newcommand{\jump}[1]{\llbracket #1\rrbracket}

\newcommand{\recon}[1]{\left[#1\right]_\e}
\newcommand{\decon}[1]{\left[#1\right]^\e}
\newcommand{\decongen}[2]{\left[#1\right]^{\e_{#2}}}

\newcommand{\eeval}[1]{\left[\frac{#1}{\e}\right]}

\newcommand{\deval}[1]{\left\{\frac{#1}{\e}\right\}}

\newcommand{\twosc}{\stackrel{2}{\rightarrow}}
\newcommand{\stwosc}{\stackrel{2-str.}{\longrightarrow}}
\newcommand{\e}{\varepsilon}
\newcommand{\p}{\varphi}
\newcommand{\R}{\mathbb{R}}
\newcommand{\N}{\mathbb{N}}
\newcommand{\Z}{\mathbb{Z}}

\newcommand{\dive}{\operatorname{div}}
\newcommand{\dist}{\operatorname{dist}}
\newcommand{\id}{\operatorname{id}}

\newcommand{\ddt}{\frac{\operatorname{d}}{\mathrm{d}t}}

\newtcolorbox{mybox}[1]{%
    tikznode boxed title,
    enhanced,
    arc=0mm,
    interior style={white},
    attach boxed title to top left= {yshift=-\tcboxedtitleheight/2-0.05cm, xshift=0.7cm},
    fonttitle=\small\bfseries,
    colbacktitle=white,coltitle=black,
    boxed title style={size=small,colframe=white,boxrule=0pt},
    title={#1}}

\theoremstyle{definition}
\newtheorem{theorem}{Theorem}[section]

\newtheorem{lemma}[theorem]{Lemma}

\newtheorem{remark}[theorem]{Remark}

\numberwithin{equation}{section}

\newcommand{\subjclass}[1]{\bigskip\noindent\emph{2010 Mathematics Subject Classification:}\enspace#1}
\newcommand{\keywords}[1]{\noindent\emph{Keywords:}\enspace#1}

\begin{document}


\baselineskip=17pt


\title{Homogenization of a moving boundary problem with prescribed normal velocity}

\author{Michael Eden\\
University Bremen, Germany\\
leachim@math.uni-bremen.de}

\date{\today}

\maketitle


\begin{abstract}
The analysis and homogenization of a moving boundary problem for a highly heterogeneous, periodic two-phase medium is considered.
In this context, the normal velocity governing the motion of the interface separating the two competing phases is assumed to be prescribed.
Parametrizing the boundary motion via a height function, the so-called \emph{Direct Mapping Method} is employed to construct a coordinate transform characterizing the changes with respect to the initial setup of the geometry.
Utilizing this transform, well-posedness of the problem is established.
After characterizing the limit behavior (with respect to the heterogeneity parameter $\e\to0$) of the functions related to the transformation, the homogenized problem of the heterogeneous two-scale problem is deduced. 


\end{abstract}

\section{Introduction}
We consider the analysis and the homogenization of a moving boundary problem that describes phase transitions occurring in highly heterogeneous two-phase media.
Here, the two phases in question are separated via a sharp interface whose exact evolution is not known at the outset.

To me more specific, let $\Omega\subset\R^3$ be a bounded domain and let $\Omega^{(1)}_\e$, $\Omega^{(2)}_\e\subset\Omega$ be $\e$-periodic subdomains representing the initial set-up of the two-phases occupying $\Omega$.
Here, the small parameter $\e$ represents the characteristic length of the inhomogeneities of the medium.
The interface between the competing phases will be denoted by $\Gamma_\e$.
Due to phase transitions, this geometrical setup might change with time leading to domains $\Omega^{(i)}_\e(t)$ ($i=1,2$) and interface $\Gamma_\e(t)$ at time $t$ which, in general, are not necessarily periodic anymore.
With $n_{\Gamma_\e}$ and $V_{\Gamma_\e}$, we denote normal vector pointing outwards $\Omega_\e^{(2)}$ and the normal velocity of $\Gamma_\e(t)$ in normal direction, respectively.

Now, let $\theta^{(i)}_\e=\theta^{(i)}_\e(t,x)$ denote the temperature in the respective domains.
In this work, we consider a two-phase heat problem accounting for latent heat and phase transitions given by 
\begin{subequations}
\begin{alignat}{2}
\partial_t\theta_\e^{(i)}-\kappa^{(i)}_\e\Delta\theta_\e^{(i)}&=f_\e^{(i)}&\quad&\text{in}\ \Omega_\e^{(i)}(t),\label{a}\\
\jump{\theta_\e}&=0&\quad&\text{on}\ \Gamma_\e(t),\label{b}\\
-\jump{\kappa_\e\nabla\theta_\e}\cdot n_\e&=LV_{\Gamma_\e}&\quad&\text{on}\ \Gamma_\e(t),\label{c}\\
V_{\Gamma_\e}&=\e v_\e&\quad&\text{on}\ \Gamma_\e(t)\label{d}
\end{alignat}
\end{subequations}
complemented with appropriate boundary and initial conditions.
The aim of this paper is twofold: $(i)$ show that this two-phase problem admits a unique local-in-time solution where the interval of existence is independent of the parameter $\e$ and $(ii)$ investigate the limit behavior $\e\to0$ thereby establishing an homogenized limit problem approximating (in some sense) the above system.

For the existence part, we rely on a particularly useful approach, which was originally introduced in \cite{H81}, which is sometimes called \emph{Direct Mapping Method} or \emph{Hanzawa transformation}, and where a specific coordinate transformation is constructed.
Please note that using this method it is not possible to consider any type of topological changes.
Regarding the limit process in the context of mathematical homogenization, we employ the notion of (strong) two-scale convergence as introduced in \cite{Al92,N89}.

Combining the analysis of moving boundary problems with the mathematical homogenization leads to significant mathematical and technical challenges.
First, the motion of the interface has to satisfy certain estimates uniformly with respect to the scale parameter $\e$.
This means that the influence of $\e$ has to be accounted for very carefully.
Second, we have to show strong two-scale convergence of some functions related to the transformation as the usual two-scale convergence is not sufficient to pass to the limit (due to the coordinate transform).

Similar moving boundary problems to the system given by \cref{a,b,c,d} without the heterogeneity parameter $\e$ were considered in, e.g., \cite{Ch92,CR92,PSS15}.
The heterogeneous case might arise in situations where the spatial scale at which we can observe such transformations is several orders of magnitude below the size of the materials itself are; typical examples would be phase transformations in porous media or in steel.
Such heterogeneous problems were considered in, e.g., \cite{E04, EKK02,H16}.

For the more general setting of a fully coupled version of System 1.1 where the normal velocity is not prescribed but rather given as a function of the temperature and the geometry of the interface, typical choices would be $v_\e=\theta_\e-\theta_{crit}$ (the law of \emph{kinetic undercooling}) or $v_\e=-H_{\Gamma_\e}+\theta_\e-\theta_{crit}$ (\emph{Gibbs-Thomson undercooling}).
Here, $\theta_{crit}$ denotes the critical temperature of the phase transition in question and $H_{\Gamma_\e}(t)$ the mean curvature function of the interface $\Gamma_\e(t)$.
One possible way to tackle such fully coupled problems is in the context of maximal parabolic regularity, see, e.g., \cite{PS16, PSR13}.
This, however, runs into additional troubles in the heterogeneous case due to the extensive $\e$-independent estimates that would need to be established; e.g., $\theta_\e(t)$ would have to be uniformly bounded in $W^{2,\infty}(\Gamma_\e(t))$.

This work can therefore be seen as an important intermediate step in the analysis of the fully coupled case.
In the existing literature regarding the homogenization of evolving microstructures, the changes in the geometry are usually assumed to be a priori known (the case of prescribed coordinate transform), see \cite{D15, EM17, P09, NM11}; a scenario which is easier to tackle.

This work is organized as follows: 
In Section~\ref{section:setting}, we introduce the $\e$-periodic geometry, the moving boundary problem with prescribed normal velocity as well as the level set equation associated with the normal velocity.
The main results regarding the moving boundary problem, \Cref{theorem:1,theorem:2,theorem:3,theorem:4}, are then given in \Cref{s:main_results}.
Finally, \Cref{s:interface_movement,mi:sec:limit} are dedicated to the detailed proofs of \Cref{theorem:1} and \Cref{theorem:3}, respectively.

\section{Setting and problem statement}\label{section:setting}
\subsection{Geometrical setup}
Let $S=(0,T)$, $T>0$, represent the time interval of interest and let $\Omega\subset\R^3$ be a bounded Lipschitz domain whose outer normal vector we denote with $\nu=\nu(x)$.
In addition, let $\e=(\e_n)_{n\in\N}$ be a monotonically decreasing sequence of positive numbers converging to zero.

Now, take open and disjoint sets $Y^{(1)},$ $Y^{(2)}\subset(0,1)^3=:Y$ such that $Y^{(1)}$ is connected, $\overline{Y^{(2)}}\subset Y$, and $Y=Y^{(1)}\cup\overline{Y^{(2)}}$.
Moreover, let $\Gamma:=\partial Y^{(2)}$ be a $C^3$-hypersurface.
By $n_\Gamma=n_\Gamma(\gamma)$, $\gamma\in\Gamma$, we denote the normal vector of $\Gamma$ pointing outwards of $Y^{(2)}$.

In order to circumvent problems due to complex structures at the boundary, we remove the boundary layer of thickness $\e$ via
$$
\widetilde{\Omega}_\e=\Omega\cap\left(\bigcup_{k\in Z_\e}\e(Y+k)\right),\quad\text{where}\ Z_\e=\{k\in\Z^3\ : \ \e(Y+k)\subset\Omega\}.
$$
Then, we introduce the $\e Y$-periodic domains $\Omega^{(i)}_\e$ ($i=1,2$) and the interface $\Gamma_\e$ representing the two phases and the phase boundary, respectively, via
$$
\Omega_\e^{(2)}=\widetilde{\Omega}_\e\cap\left(\bigcup_{k\in \Z^3}\e(Y+k)\right),\quad \Omega_\e^{(1)}=\Omega\setminus\overline{\Omega_\e^{(2)}},\quad\Gamma_\e=\partial\Omega_\e^{(2)}.
$$
Note that, by design, $\partial\Omega_\e^{(1)}=\partial\Omega$ and $\dist(\partial\Omega,\Gamma_\e)\geq\e$.

With $t\mapsto\Gamma_\e(t)$ and $t\mapsto\Omega_\e^{(i)}(t)$ for $t\in S$, we denote the evolution of the interface and the domains, respectively.
We set
$$
Q_\e^{(i)}:=\bigcup_{t\in S}\{t\}\times\Omega_\e^{(i)}(t),\qquad\Xi_\e:=\bigcup_{t\in S}\{t\}\times\Gamma_\e(t).
$$
Finally, we assume the overall domain $\Omega$ to be time-independent; that is $\Omega=\Omega_\e^{(1)}(t)\cup\Omega_\e^{(2)}(t)\cup\Gamma_\e(t)$ for all $t\in S$.
An illustration of the general geometrical setup is given via \Cref{figure1}.

\begin{figure}[h]
\centering 
\includegraphics[width=\textwidth,]{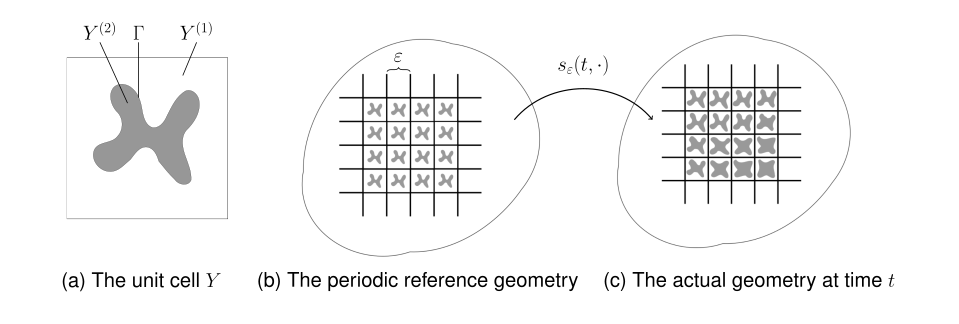}
\caption{Illustration of the geometrical setup.
Here, the motion function $s_\e(t,\cdot)$ characterizes the changes in geometry.}
\label{figure1}
\end{figure}

As a $C^3$-hypersurface, $\Gamma$ admits a tubular neighborhood $U_\Gamma$ of width $a>0$.
Moreover, the function
$$
\Lambda\colon\Gamma\times(-a,a)\to U_\Gamma,\quad \Lambda(\gamma,s):=\gamma+sn_\Gamma(\gamma)
$$
is a $C^2$-diffeomorphism satisfying $\Lambda(\Gamma\times(-a,a))\subset Y$; we refer to \cite[Section 3.1, p.65]{PS16}.
Similarly, we introduce the $\e$-scaled $C^2$-diffeomorphism
$$
\Lambda_\e\colon\Gamma_\e\times(-\e a,\e a)\to U_{\Gamma_\e},\quad \Lambda_\e(\gamma,r)=\gamma+rn_{\Gamma_\e}(\gamma).
$$
the family of interfaces
\begin{equation}\label{mi:eq:faminter}
\Gamma_\e^{(l)}:=\left\{\Lambda_\e(\gamma,l)\ : \ \gamma\in\Gamma_\e\right\}\quad \text{for}\ \ l\in\left[-\e a,\e a\right],
\end{equation}
and the family of tubes around $\Gamma_\e$
$$
U_{\Gamma_\e}(r):=\bigcup_{l\in(-\e ra,\e ra)}\Gamma_\e^{(l)}\qquad\left(r\in(0,1]\right).
$$
We set $U_{\Gamma_\e}=U_{\Gamma_\e}(1)$.
For $\gamma\in\Gamma_\e$, let $L_{\Gamma_\e}(\gamma)=-\nabla_{\Gamma_\e}n_{\Gamma_\e}(\gamma)$ denote the \emph{Weingarten map}, where we have (\cite[Section 2.1]{PS16})
%
\begin{align}\label[in]{weingarten}
\sup_{\gamma\in\Gamma_\e}|L_{\Gamma_\e}(\gamma)|\leq\frac{1}{2\e a}.
\end{align}
For $l\in[-\e a,\e a]$ and $\gamma\in\Gamma_\e^{(l)}$, the normal vector of the interface $\Gamma_\e^{(l)}$ in $\gamma$ is given as $n_{\Gamma_\e}(P_{\Gamma_\e}(\gamma))$, where $P_{\Gamma_\e}\colon U_{\Gamma_\e}\to\Gamma_\e$ denotes the projection operator.
The inverse of $\Lambda_\e$ is given via
$$
\Lambda_\e^{-1}\colon U_{\Gamma_\e}\to\Gamma_\e\times[-\e a,\e a],\quad\Lambda_\e^{-1}(x)=\left(P_{\Gamma_\e}(x),d_{\Gamma_\e}(x)\right)^T.
$$
Here, $d_{\Gamma_\e}\colon U_\Gamma\to\R$ is the signed distance function for $\Gamma_\e$, i.e., 
$$
d_{\Gamma_\e}(x)=\begin{cases}\dist(x,\Gamma_\e),\quad& x\in U_{\Gamma_\e}\setminus\Omega_\e^{(2)}\\ -\dist(x,\Gamma_\e),\quad& x\in U_{\Gamma_\e}\cap\Omega_\e^{(2)}\end{cases}.
$$
%
\subsection{Problem statement}
For $k,l\in\N$, we introduce the Sobolev space
$$
W^{(k,l),\infty}(S\times\Omega)=\left\{u\in L^\infty(S\times\Omega) :\partial_t^iu,D_x^ju\in L^\infty(S\times\Omega)\ (1\leq i\leq k,\ 1\leq j\leq l)\right\}
$$
and note that $W^{(k,k),\infty}(S\times\Omega)=W^{k,\infty}(S\times\Omega)$.

Now, take $\theta_\e^{(i)}=\theta_\e^{(i)}(t,x)$ ($i=1,2$) to represent the temperature in the respective domains $Q_\e^{(i)}$.
In the following, we consider the moving boundary problem given by:
\begin{mybox}{Moving boundary problem with prescribed normal velocity}
\begin{subequations}\label{eq:problem}
\begin{alignat}{2}
\partial_t\theta_\e^{(1)}-\kappa^{(1)}\Delta\theta_\e^{(1)}&=f_\e^{(1)}&\quad&\text{in}\ Q_\e^{(1)},\label{eq:problem:1}\\
\partial_t\theta_\e^{(2)}-\e^2\kappa^{(2)}\Delta\theta_\e^{(2)}&=f_\e^{(2)}&\quad&\text{in}\ Q_\e^{(2)},\label{eq:problem:2}\\
\theta_\e^{(1)}&=\theta_\e^{(2)}&\quad&\text{on}\ \Xi_\e,\label{eq:problem:3}\\
-\left(\kappa^{(1)}\nabla\theta_\e^{(1)}-\e^2\kappa^{(2)}\nabla\theta_\e^{(2)}\right)\cdot n_\e&=LV_{\Gamma_\e}&\quad&\text{on}\ \Xi_\e,\label{eq:problem:4}\\
V_{\Gamma_\e}&=\e v_\e&\quad&\text{on}\ \Xi_\e,\label{eq:problem:5}\\
-\kappa^{(1)}\nabla\theta_\e^{(1)}\cdot\nu&=0&\quad&\text{on}\ S\times\partial\Omega,\label{eq:problem:6}\\
\theta_\e^{(1)}(0)&=\vartheta_\e^{(1)}&\quad&\text{in}\ \Omega_\e^{(1)},\label{eq:problem:7}\\
\theta_\e^{(2)}(0)&=\vartheta_\e^{(2)}&\quad&\text{in}\ \Omega_\e^{(2)}.\label{eq:problem:8}
\end{alignat}
\end{subequations}
\end{mybox}
Here, the positive constants $\kappa^{(i)}$ denote the heat conductivity coefficients and $L$ denotes the constant of latent heat.
The actual mathematical problem connected to this system is as follows:
Given volume heat source densities $f_\e^{(i)}\colon Q_\e^{(i)}\to\R$, a function $v_\e\colon\Xi_\e\to\R$ governing the movement of the interface, and initial values $\vartheta_\e^{(i)}\colon\Omega_\e^{(i)}\to\R$, find the corresponding evolution of the domains, i.e., find $\Omega_\e^{(i)}(t)$ and $\Gamma_\e(t)$ for all $t\in S$, and the temperature functions $\theta_\e^{(i)}\colon Q_\e^{(i)}\to(0,\infty)$ such that all equations of the above system are satisfied.

Now, let $v_\e\in W^{(1,2),\infty}(S\times\Omega)$ be the outward normal velocity of the moving interface $\Gamma_\e(t)$.
Let us assume that the corresponding motion of $\Gamma_\e$ can be described via a regular $C^1$-motion. 
Then, there exists a level set function $\p_\e\colon S\times\Omega\to\R$ such that
\begin{align*}
\Gamma_\e(t)&=\left\{x\in\Omega \ : \ \p_\e(t,x)=0\right\},\\
|\nabla\p_\e(t,x)|&>0\quad \text{on}\ \Xi_\e,\\
\p_\e(t,x)&<0\quad \text{on}\ \partial\Omega.
\end{align*}
The normal velocity $\e v_\e$ and the level set function $\p_\e$ are connected via (\cite[Section 4.1]{OF02})
\begin{equation*}
\partial_t\p_\e=\e|\nabla\p_\e|v_\e \quad\text{on}\ \Xi_\e.
\end{equation*}
Based on these geometric considerations, we formulate the motion problem as a level set problem:
\begin{mybox}{Motion problem via level set equation}
Find $\p_\e\in C^1(S\times\Omega)$ such that
\begin{subequations}
\begin{align}
\partial_t\p_\e&=\e|\nabla\p_\e|v_\e\quad\text{on}\ \Xi_\e,\label{eq:level_set_1}\\
|\nabla\p_\e(t,x)|&>0\quad \text{on}\ \Xi_\e,\label{eq:level_set_2}\\
\frac{\partial_t\p_\e-\e|\nabla\p_\e|v_\e}{\p_\e}&\in W^{(0,1),\infty}(S\times\Omega)\label{eq:level_set_3},\\
\Gamma_\e&=\{x\in\Omega \ : \ \p_\e(0,x)=0\},\label{eq:level_set_4}\\
\Omega_\e^{(1)}&=\{x\in\Omega \ : \ \p_\e(0,x)<0\}.\label{eq:level_set_5}
\end{align}
\end{subequations}
\end{mybox}
The family of sets $(\Gamma_\e(t))_{t\in S}$ defined via 
\begin{equation*}
\Gamma_\e(t)=\{x\in\Omega\ : \ \p_\e(t,x)=0\}\label{eq:level_set_6}
\end{equation*}
is called the solution of the motion problem.
The condition \eqref{eq:level_set_3} is a shorthand for: the function $\frac{\partial_t\p_\e-\e|\nabla\p_\e|v_\e}{\p_\e}\colon (S\times\Omega)\setminus\Xi_\e\to\R$ can be extended to a function in $W^{(0,1),\infty}(S\times\Omega)$.
Note that this condition is merely technical in that it is not needed for the level set function $\p_\e$ to correspond to the motion of the interface; it is, however, needed in \Cref{mi:lem:solfbp}.

We also point out that uniqueness of a solution of the motion problem only asserts uniqueness of the the family of hypersurfaces $(\Gamma_\e(t))_{t\in S}$ but not uniqueness of the level set function $\p_\e$.
Indeed, for every $\alpha>0$, $\alpha\p_\e$ corresponds to the same motion problem.

\section{Main results}\label{s:main_results}
In this section, we present the main results. 
As some of the proofs are fairly long and technical, they are postponed to subsequent chapters: \Cref{s:interface_movement} and \Cref{mi:sec:limit} are devoted to the proofs of \Cref{theorem:1} and \Cref{theorem:3}, respectively.

We start by formulating the requirements for the data (normal velocity, source densities, and initial values) that are needed to ensure the well-posedness of the microscopic problems as well as to facilitate the passage $\e\to0$.
%
\begin{itemize}
\item[(\textbf{A1})]\label{mi:ass:1} Let $v_\e\in W^{(1,3),\infty}(S\times\Omega)$ with $\mathrm{supp}(v_\e)\subset U_{\Gamma_\e}$ and
\begin{align*}
l_v:=\sup_{\e>0}\left(\|v_\e\|_{W^{1,\infty}(S\times\Omega )}+\e\|D^2_x v_\e\|_\infty+\e^2\|D^3_xv_\e\|_\infty\right)<\infty.
\end{align*}
\item[(\textbf{A2})]\label{mi:ass:2}
For $i=1,2$, let $f_\e^{(i)}\in L^2(Q_\e^{(i)})$ and $\vartheta_\e^{(i)}\in L^2(\Omega_\e^{(i)})$ such that 
$$
\sup_{\e>0}\left(\|f_\e^{(i)}\|_{L^2(Q_\e^{(i)})}+\|\vartheta_\e^{(i)}\|_{L^2(\Omega_\e^{(i)})}\right)<\infty.
$$
\item[(\textbf{A3})]\label{mi:ass:3}
There is a function $v\in L^2(S\times\Omega;W^{1,2}_\#(Y))^3$ satisfying
\begin{alignat*}{2}
\decon{v_\e}\to v,\ \ \decon{Dv_\e}\to D_yv,\ \ \e\decon{D^2v_\e}\to D_y^2v\quad\text{in}\ L^2(S\times\Omega\times Y)^{3}.
\end{alignat*}
\end{itemize}

Here, $\decon{v_\e}\colon S\times\Omega\times Y\to\R$ is the periodic unfolding of $v_\e\colon S\times\Omega\to\R$ defined via $\decon{v_\e}(t,x,y)=v\left(t,\e y+\e\left[\frac{x}{\e}\right]\right)$ where $\left[x\right]$ denotes the unique $k\in\Z^3$ for which $x-k\in[0,1)^3$; for details, we refer to \cite{CDG02}.
Furthermore, the number sign subscript $\#$ indicates spaces of periodic functions:
$$
W^{1,2}_\#(Y)=\{u\in W^{1,2}_\mathrm{loc}(\R^3) \ : \ u_{|Y}\in W^{1,2}(Y),\ u(y)=u(e_j+y)\ \text{for a.a.}\ y\in Y\ (j=1,2,3)\}.
$$
If $\decon{v_\e}\to v$ in $L^2(S\times\Omega\times Y)$, we say that $v_\e$ strongly two-scale converges to $v$ ($v_\e\stwosc v$); if $\decon{v_\e}\rightharpoonup v$, we say that $v_\e$ two-scale converges to $v$ ($v_\e\twosc v$).
The correspondence of this notion to the usual definition of two-scale convergence (see \cite{Al92}) can  be found, e.g., in \cite{CDG08}. 

The regularity and the estimates postulated via Assumption (A1) ensure well-posedness of the motion problem given by \cref{eq:level_set_1,eq:level_set_2,eq:level_set_3,eq:level_set_4,eq:level_set_5} and the validity of corresponding a priori estimates.
With Assumption (A2), these results can be used to tackle the heat problem given by \cref{eq:problem:1,eq:problem:2,eq:problem:3,eq:problem:4,eq:problem:5,eq:problem:6,eq:problem:7,eq:problem:8}).
Finally, Assumption (A3) is necessary for the homogenization process.

The following two results, namely, \Cref{theorem:1} and \Cref{theorem:3}, are the cornerstones of this work; their proofs are given in \Cref{s:interface_movement} and \Cref{mi:sec:limit}, respectively.

\begin{theorem}\label{theorem:1}
Under Assumption (A1), there is $T_v=T(l_v)\in S$, which is independent of $\e>0$, and a function $h_\e\colon[0,T_v]\times\Gamma_\e\to(-\e a,\e a)$ such that
$$
\Gamma_\e(t)=\left\{\gamma+h_\e(t,\gamma)n_{\Gamma_\e}(\gamma) \ : \ \gamma\in\Gamma_\e\right\}\quad (t\in[0,T_v]).
$$
The time $T_v$ is increasing for decreasing values of $l_v$ and we have $(0,T_v)=S$ for sufficiently small $l_v>0$.
Also, there is a corresponding, regular $C^1$-motion $s_\e\colon[0,T_v]\times\overline{\Omega}\to\overline{\Omega}$ satisfying $s_\e(0)=\id$, $s_\e(t,\Omega_\e^{(i)})=\Omega_\e^{(i)}(t)$ ($i=1,2$), and
$$
\|Ds_\e\|_\infty\leq2,\quad\|\left(Ds_\e\right)^{-1}\|_\infty\leq2.
$$
\end{theorem}
\begin{proof}
This follows via \Cref{mi:thm:height} and \Cref{mi:lem:psi}.
The statements and proof of these results are given in \Cref{s:interface_movement}. 
\end{proof}
In the following, we set $S_v=(0,T_v)$.
\begin{theorem}\label{theorem:3}
Under Assumptions (A1) and (A2), there is $s\in L^\infty(S_v\times\Omega\times Y)$ with $\partial_ts$, $D_ys\in L^\infty(S_v\times\Omega\times Y)$ such that $Ds_\e\stwosc D_ys$.
\end{theorem}
\begin{proof}
The proof of this theorem is given in \Cref{mi:sec:limit}, see \Cref{mi:lem:psilimit}.
\end{proof}

Using the results given in \Cref{theorem:1,theorem:3}, it is then possible to investigate the associated heat conduction problem:
\begin{theorem}\label{theorem:2}
Under Assumptions (A1) and (A2), there is a unique solution of the mathematical problem corresponding to the system given via \cref{eq:problem:1,eq:problem:2,eq:problem:3,eq:problem:4,eq:problem:5,eq:problem:6,eq:problem:7,eq:problem:8}.
In addition, we find that
\begin{align*}
\sup_{\e>0}\left(\|\theta_\e\|^2_{L^\infty(S_v;L^2(\Omega))}+\|\nabla\theta_\e^{(1)}\|^2_{L^2(S_v\times\Omega_\e^{(1)})}+\e^2\|\nabla\theta_\e^{(2)}\|^2_{L^2(S_v\times\Omega_\e^{(2)})}\right)<\infty
\end{align*}
\end{theorem}
\begin{proof}
Using the transformation function $s_\e$ (given via \Cref{theorem:1}) to arrive at a fixed-domain formulation of the problem, we are almost exactly in the situation described in~\cite{EM17} (without the mechanical part).
\end{proof}

We set $Q_Y=\bigcup_{(t,x)\in S_v\times\Omega}\{(t,x)\}\times Y^{(2)}(t,x)$. 
With $\mathds{1}_E$, we denote the indicator function of a set $E$.
\begin{theorem}\label{theorem:4}
Let Assumptions (A1)--(A3) hold.
There are functions $\theta\in L^2(S_v;W^{1,2}(\Omega))$ and $\theta^{(2)}\in L^2\left(Q_Y\right)$, where $\theta^{(2)}(t,x,\cdot)\in W^{1,2}(Y^{(2)}(t,x))$ for almost all $(t,x)\in S_v\times\Omega$, such that 
$$
\mathds{1}_{\Omega_\e^{(1)}}\theta_\e^{(1)}\rightharpoonup|Y^{(1)}(t,x)|\theta,\ \ \mathds{1}_{\Omega_\e^{(2)}}\theta_\e^{(2)}\rightharpoonup\int_{Y^{(2)}(t,x)}\theta^{(2)}\di{y}
\quad\text{in}\ L^2(S\times\Omega).
$$
Moreover, they solve the following homogenized distributed microstructure problem:
The macroscopic temperature $\theta$ is governed by an effective heat conduction problem given via
\begin{subequations}
\begin{alignat}{2}
\partial_t\theta-\dive(\kappa^{h}\nabla\theta)&=f^{h}+f_\Gamma^h&\quad&\text{in}\ \ S_v\times\Omega,\\
-\kappa^{h}\nabla\theta\cdot\nu&=0&\quad&\text{on}\ \  S_v\times\partial\Omega,\\
\theta(0)&=\vartheta^{h}&\quad&\text{in}\ \Omega,
\end{alignat}
which is coupled, via the Dirichlet boundary condition \eqref{mi:eq:diri}, to a micro heat problem with time dependent microstructures for $\theta^{(2)}$ in the form of
\begin{alignat}{2}
\partial_t\theta^{(2)}-\kappa^{(2)}\Delta_y\theta^{(2)}&=f^{(2)}&\quad&\text{in}\ Y^{(2)}(t,x),\, t\in S_v,\, x\in\Omega,\\
\theta^{(2)}&=\theta&\quad&\text{on}\ \Gamma(t,x),\, t\in S_v,\, x\in\Omega,\label{mi:eq:diri}\\
\theta^{(2)}(0)&=\vartheta^{(2)}&\quad&\text{in}\ \Omega\times Y^{(2)}.
%
\intertext{Finally, the motion of the interface $\Gamma(t,x)$ in normal direction is governed by}
V_{\Gamma}&=v&\quad&\text{on}\ \Gamma(t,x),\, t\in S_v,\, x\in\Omega.
\end{alignat}
\end{subequations}
Here, the effective coefficients are given as
\begin{alignat*}{2}
f^{h}&=\int_{Y^{(1)}(t,x)}f^{(1)}\di{y},&\quad f_\Gamma&=\int_{\Gamma(t,x)}Lv+\kappa^{(2)}\nabla_y\theta^{(2)}\cdot n\di{\sigma},\\
\vartheta^{h}&=\int_{Y^{(1)}(t,x)}\vartheta^{(1)}\di{y},&\quad
\left(\kappa^{h}\right)_{ij}&=\kappa^{(1)}\min_{\tau\in W^{1,2}(Y^{(1)}(t,x))}\int_{Y^{(1)}(t,x)}\left(\nabla_y\tau+e_j\right)\cdot e_i\di{y},
\end{alignat*}
and $f^{(i)}$, $\vartheta^{(i)}$ ($i=1,2$), and $v$ are the two-scale limits of their corresponding $\e$-counterparts.
\end{theorem}
\begin{proof}
Due to the strong convergence result of \Cref{mi:lem:psilimit}, this homogenization results follows via a standard two-scale limit procedure and is a special case of the homogenization of the thermoelasticity problem performed in \cite{EM17}.
\end{proof}

\section{Interface motion (proof of \texorpdfstring{\Cref{theorem:1}}{Theorem 6.1})}\label{s:interface_movement}

This section is devoted to the proof of \Cref{theorem:1}.
As a short guideline, this proof follows the following strategy:
\begin{itemize}
	\item[$(i)$] We investigate a nonlinear, parametrized ODE-system -- given by \cref{s:ode:1,s:ode:2,s:ode:3,s:ode:4} -- tracking the interface motion.
	This is done via \Cref{lemma:estimates:yz,mi:lem:tdelta}.
	\item[$(ii)$] We then show that the motion problem given via conditions~\eqref{eq:level_set_1}-\eqref{eq:level_set_5} has a unique solution; see \Cref{mi:lem:solfbp}.
	\item[$(iii)$] In \Cref{mi:thm:height}, the local-in-time existence of the height function $h_\e$ is then deduced via the implicit function theorem.
	\item[$(iv)$] Finally, we construct a family of $C^1$-diffeomorphisms $s_\e(t,\cdot)\colon\overline{\Omega}\to\overline{\Omega}$ and investigate its properties; see \Cref{mi:lem:psi}.
\end{itemize}

The first two steps can be found in \Cref{mi:ssec:fbp}, and steps $(iii)$ and $(iv)$ are the topic of \Cref{mi:ssec:motion}.
In the following, we take $C>0$ to denote any generic constant that is independent of both $l_v$ and $\e$ (but may depend on the interface $\Gamma=\partial Y^{(2)}$ as well as the overall domain $\Omega$).
In addition, we take $C(l_v)$ (sometimes with a subscript, e.g., $C_w(l_v)$) to denote the value at $l_v$ of any monotonically increasing, continuous, and $\e$-independent function $C\colon[0,\infty)\to(0,\infty)$. 

Note that this section is structurally similar to \cite[Section 3]{Ch92}, where the main substantial differences are due to the parameter $\e$ and its role in the context of homogenization.

\subsection{Interface motion problem}\label{mi:ssec:fbp}
We consider the following nonlinear ODE system:
\begin{mybox}{ODE system describing the interface motion}
Find $y_\e, z_\e\colon S\times U_{\Gamma_\e}\to\R^3$ such that
\begin{subequations}\label{s:ode}
\begin{alignat}{2}
\partial_ty_\e(t,x)&=-\e\frac{z_\e(t,x)}{|z_\e(t,x)|}v_\e(t,y_\e(t,x))&\quad&\text{in}\ S\times U_{\Gamma_\e}\label{s:ode:1},\\
\partial_tz_\e(t,x)&=\e|z_\e(t,x)|\nabla v_\e(t,y_\e(t,x))&\quad&\text{in}\ S\times U_{\Gamma_\e},\label{s:ode:2}\\
y_\e(0,x)&=x&\quad&\text{in}\ U_{\Gamma_\e},\label{s:ode:3}\\
z_\e(0,x)&=-n_{\Gamma_\e}(P_{\Gamma_\e}x)&\quad&\text{in}\ U_{\Gamma_\e}.\label{s:ode:4}
\end{alignat}
\end{subequations}
\end{mybox}
%
We extend every solution $y_\e$ to all of $\Omega$ by setting $y_\e(t,x)=x$.
Due to $\mathrm{supp}\,v_\e\subset U_{\Gamma_\e}$, $y_\e$ is then continuous across $\partial U_{\Gamma_\e}$.

\begin{remark}
In \Cref{mi:lem:solfbp}, we show that the function $y_\e$ characterizes the interface motion in the sense that $\Gamma_\e(t)=y_\e(t,\Gamma_\e)$. 
The function $z_\e$ describes the direction of the motion.
This is illsutrated in \Cref{mi:fig:tracking}.
Note that, if $\nabla v_\e\equiv0$, the solution satisfies $y_\e(t,\gamma)=\gamma+d_{\Gamma_\e}(y_\e(t,\gamma))n_{\Gamma_\e}(\gamma)$ for all $\gamma\in\Gamma$.
\end{remark}

\begin{figure}[h]
\centering
\begin{tikzpicture}[scale=1]
	\pgftext{\includegraphics[width=.5\textwidth,]{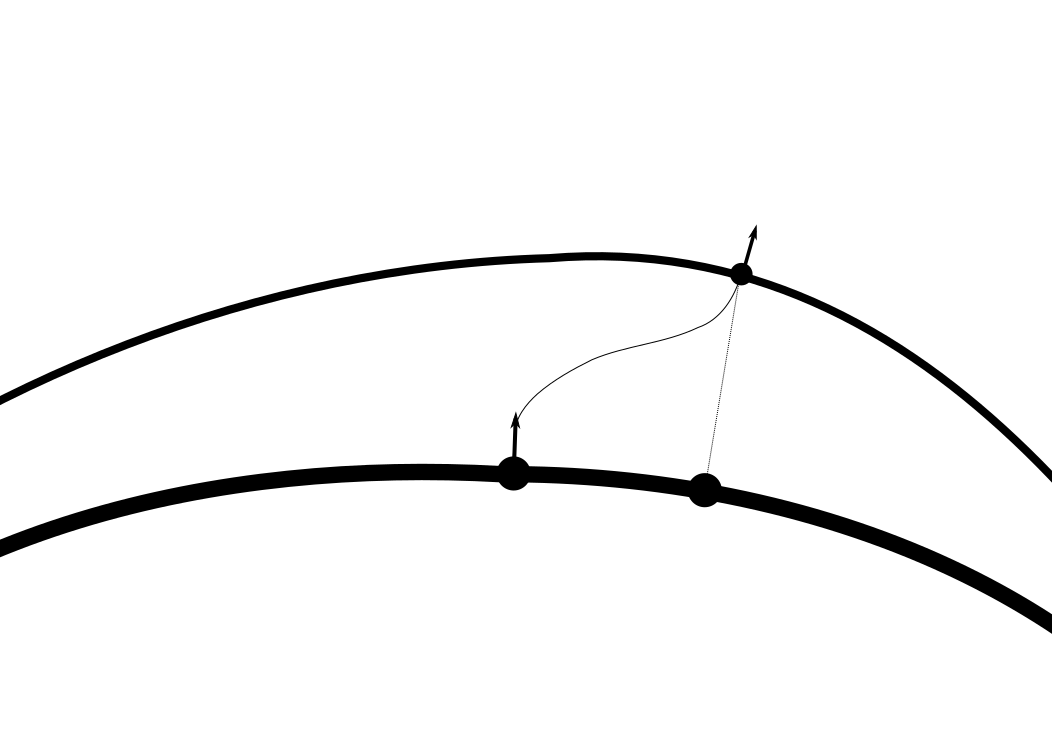}};
	\draw[thick] (-2.8,-1) -- (-3.2,-1.7);
	\draw (-3.1,-1.9) node[left] {{$\Gamma_\e$}};
	\draw[thick] (-2.8,0.32) -- (-3.2,0.8);
	\draw (-3.2,0.7) node[above] {{$\Gamma_\e(t)$}};
	\draw (-0.2,-0.9) node[below] {{$\gamma$}};
	\draw (-0.64,0) node[below] {{{\footnotesize$n_{\Gamma_\e}(\gamma)$}}};
	\draw (1.3,-1.1) node[below] {{{\footnotesize$P_{\Gamma_\e}(y_\e(t,\gamma))$}}};
	\draw (1.2,1.5) node[below] {{{\footnotesize $z_\e(t,\gamma)$}}};
	\draw (2.4,1.27) node[below] {{{\footnotesize$y_\e(t,\gamma)$}}};
\end{tikzpicture}
\caption[Tracking of the interface motion]{Part of the surface $\Gamma_\e$ and its position at time $t$, $\Gamma_\e(t)$. The function $y_\e$ characterizes the motion by tracking the paths of the material points. As an example, we see the path of $y_\e$ for $\gamma=y_\e(0,\gamma)$ over the interval $(0,t)$.
In addition, we see the change in the normal vector from $n_{\Gamma_e}(\gamma)=z_\e(0,\gamma)$ to $z_\e(t,\gamma)$.
The goal is to find the corresponding height function $h_\e$ that satisfies $h_\e(P_{\Gamma_\e}(y_\e(t,\gamma)))=d_{\Gamma_\e}(y_\e(t,\gamma))$.}
\label{mi:fig:tracking}
\end{figure}
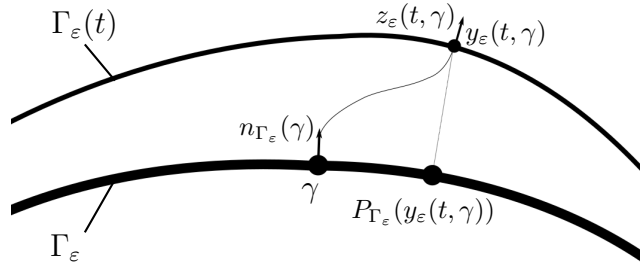

We introduce functions
\begin{align*}
f_\e&\colon\overline{S}\times\left(\R^3\times\R^3\setminus\{0\}\right)\to\R^3\times\R^3,\quad f_\e(t,(y,z))=\left(\frac{z}{\left|z\right|}v_\e(t,y),\left|z\right|\nabla v_\e(t,y)\right)^T,\\
g_\e&\colon\Omega\to\R^3\times\R^3,\quad g_\e(x)=\left(x,-n_{\Gamma_\e}(P_{\Gamma_\e}(x))\right)^T.
\end{align*}
Setting $w_\e=(y_\e,z_\e)^T$, \cref{s:ode:1,s:ode:2,s:ode:3,s:ode:4} then become
\begin{subequations}\label{s:ode_short}
\begin{alignat}{2}
\partial_tw_\e(t,x)&=\e f_\e\left(t,w_\e(t,x)\right)&\quad&\text{in}\ S\times U_{\Gamma_\e},\\
w_\e(0,x)&=g_\e(x)&\quad&\text{in}\ U_{\Gamma_\e}.
\end{alignat}
\end{subequations}
%

\begin{lemma}\label{lemma:estimates:yz}
Let Assumption (A1) hold.
The ODE system given via \cref{s:ode:1,s:ode:2,s:ode:3,s:ode:4} admits a unique solution $(y_\e,z_\e)\in W^{(1,2),\infty}(S\times U_{\Gamma_\e})^6$.
Additionally, there exists a monotonically increasing, continuous function $C_w\colon[0,\infty)\to(0,\infty)$, which is independent of the parameter $\e$, such that
\begin{align*}
\|D_xy_\e-\mathds{I}\|_\infty+\|\partial_tD_xy_\e\|_\infty+\e\|D_x^2y_\e\|_\infty&\leq l_vC_w(l_v),\\
\e\|D_xz_\e\|_\infty+\e^2\|D_x^2z_\e\|_\infty&\leq C_w(l_v).
\end{align*} 
\end{lemma}
\begin{proof}
\emph{(i) Existence and Uniqueness.} Due to the embedding $W^{k,\infty}(U_{\Gamma_\e})=C^{k-1,1}(U_{\Gamma_\e})$ ($k\geq1$) (we refer to \cite[Theorem 7]{HKT08}) we have $v_\e,\partial_j v_\e\in C^{1,1}(S\times U_{\Gamma_\e})$ ($j=1,2,3$) which, in turn, implies $f_\e\in C^{1,1}\left(S\times\left(\R^3\times K\right)\right)$ for every compact set $K\subset\R^3\setminus\{0\}$.
Therefore, for every $x\in\Omega$, \emph{Picard-Lindeloef's existence theorem} (\cite[Proposition 1.8]{Z86}) guarantees the existence of a time $t_\e(x)\in S$ and a unique solution $w_\e(\cdot,x)=(y_\e(\cdot,x), z_\e(\cdot,x))^T\in C^{1,1}([0,t_\e(x)])^6$.
Note that $|z_\e(0,x)|=1$ independently of $x\in U_{\Gamma_\e}$.
Taking a look at \cref{s:ode:2}, we see that
$$
-\e t l_v\leq\int_0^t\frac{\partial_t(z_\e\cdot e_j)}{|z_\e|}\di{\tau}\leq\e tl_v\quad (j=1,2,3).
$$
The norm of every solution $z_\e$ is therefore bounded from below and above via
\begin{equation}\label[in]{est_zeps}
e^{-\e l_vt}\leq |z_\e(t,x)|\leq e^{\e l_v t}.
\end{equation}
As a consequence, a blow up due to $|z_\e|\to0$ is not possible in finite time and we can extend to $w_\e(\cdot,x)\in C^{1,1}(\overline{S})^6$ for $x\in U_{\Gamma_\e}$.
\\[-0.2cm]

\emph{(ii) Regularity and Estimates.} For any $x_1,x_2\in U_{\Gamma_\e}$, we find that
$$
w_\e(t,x_1)-w_\e(t,x_2)=g_\e(x_1)-g_\e(x_2)+\int_0^tf_\e(\tau,w_\e(\tau,x_1))-f_\e(\tau,w_\e(\tau,x_1))\di{\tau}.
$$
From $g_\e\in C^2(U_{\Gamma_\e})$, the Lipschitz continuity of $f_\e$ as well as $Df_\e$, and Gronwall's inequality, we can infer $w_\e(t,\cdot)\in W^{(1,2),\infty}(S\times U_{\Gamma_\e})^6$.

In the following, let $\e>0$ be sufficiently small such that $\nicefrac{1}{\sqrt{2}}\leq\|z_\e\|_\infty\leq\sqrt{2}$ (cf.~\cref{est_zeps}).
Differentiating the ODE with respect to $x\in U_{\Gamma_\e}$, we get
\begin{align}\label{eq:diffa}
\partial_tDw_\e(t,x)=\e D_x\left(f_\e\left(t,w_\e(t,x)\right)\right).
\end{align}
We define $A_\e\colon S\times\left(\R^3\times\R^3\setminus\{0\}\right)\to\R^{6\times6}$ via
\begin{align*}
A_\e(t,(y,z)):=D_{(y,z)}f_\e\left(t,(y,z)\right)
&=\begin{pmatrix} 
\frac{z}{|z|}\otimes \nabla v_\e(t,y) &v_\e B(z) \\
|z|D^2v_\e(t,y)&\nabla v_\e(t,y)\otimes\frac{z}{|z|}\end{pmatrix},
\end{align*}
where $B\colon\R^3\setminus\{0\}\to\R^{3\times3}$ is given via
\begin{equation}\label{eq:def:B}
B(z)=D\left(z\mapsto\frac{z}{|z|}\right)=
\frac{1}{|z|^{3}}\begin{pmatrix}
z_2^2+z_3^2& -z_1z_2&-z_1z_3\\
-z_1z_2& z_1^2+z_3^2&-z_2z_3\\
-z_1z_3& -z_2z_3&z_1^2+z_2^2
\end{pmatrix}.
\end{equation}
\Cref{eq:diffa} can be rewritten into
\begin{equation}\label{mi:eq:diffa:alter}
\partial_tDw_\e(t,x)=\e A_\e(t,w_\e(t,x))Dw_\e(t,x).
\end{equation}
With the estimate $\left\|B(z)\right\|\leq\nicefrac{\sqrt{2}}{|z|}$ (Frobenius-Norm), the estimate for $z_\e$ given by \cref{est_zeps}, and Assumption (A1), we get (for sufficiently small $\e$)
\begin{equation}\label[in]{est_Aeps}
\e|A_\e(t,(y_\e,z_\e)|\leq l_v(3\e+\sqrt{2})\leq 2l_v.
\end{equation}
For the initial values of the \emph{Jacobian} matrices, we have (for the derivative of $n_{\Gamma_\e}(P_{\Gamma_\e}(x))$, we refer to \cite[Chapter 2, Section 3.1]{PS16})
\begin{align*}
Dy_\e(0,x)&=\mathds{I}_3,\\
Dz_\e(0,x)&=D\left(n_{\Gamma_\e}(P_{\Gamma_\e}(x))\right)=-L_{\Gamma_\e}(P_{\Gamma_\e}(x))\left(\mathds{I}-d_{\Gamma_\e}(x)L_{\Gamma_\e}(P_{\Gamma_\e}(x))\right)^{-1}.
\end{align*}
As $|Dz_\e(0,x)|\leq\nicefrac{C}{\e}$ for some $C>0$, we can deduce estimate via \emph{Gronwall's inequality} that
\begin{align}\label[in]{ie:D_xw}
\e\left|Dw_\e(t,x)\right|\leq C\exp(2Tl_v)=:C_1(l_v).
\end{align}
For $y_\e$, we have
\begin{equation}\label{eq:xdery}
Dy_\e(t,x)
=\mathds{I}_3+\e\int_0^t\left(A_\e^{(11)}(t,w_\e(\tau,x))Dy_\e(\tau,x)+A_\e^{(12)}(t,w_\e(\tau,x))Dz_\e(\tau,x)\right)\di{\tau}.
\end{equation}
Inserting the estimate given in \cref{ie:D_xw} into \cref{eq:xdery}, we see that
\begin{align}\label[in]{ie:D_xy}
|Dy_\e(t,x)|\leq1+3TC_1(l_v)l_v.
\end{align}
Looking at \cref{eq:xdery} and using the estimate for $A_\e$ (cf.\,\cref{est_Aeps}), we get (for small $\e$)
\begin{align}\label[in]{ie:D_xty}
|\partial_tDy_\e(t,x)|&\leq3C_1(l_v)l_v.
\end{align}
Similarly, differentiating $Dw_\e$ with respect to $x_j$ ($j=1,2,3$) and estimating the different terms accordingly, we also get
\begin{align}\label[in]{ie:D_x2w}
\e^2|\partial_jDw_\e(t,x)|\leq C_2(l_v).
\end{align}
With this estimate, we can further bound $|\partial_{x_i}Dy_\e(t,x)|$ via
\begin{align}\label[in]{ie:D_x2y}
\left|\partial_{x_i}Dy_\e(t,x)\right|\leq l_vC_3(l_v).
\end{align}
The details regarding these calculations are given in~\cite[Lemma 6.6]{E18}.
Now, combining \cref{ie:D_xw,ie:D_xy,ie:D_xty,ie:D_x2w,ie:D_x2y}, the function $C_w$ can be directly constructed via $C_j(l_v)$ ($j=1,2,3$).
\end{proof}
Note that $l_vC_w(l_v)\to0$ for $l_v\to0$.
In the following lemma, we show that $y_\e(t,\cdot)$ is a homeomorphism (a minimal requirement for it to correspond to a meaningful transformation) for $t\in\overline{S}$ small enough.
Moreover, for small $l_v$, this holds for all $t\in\overline{S}$.
\begin{lemma}\label{mi:lem:tdelta}
There is a monotonically decreasing and continuous function $\delta\colon(0,\infty)\to(0,\infty)$ (we set $t_v=\min\{\delta(l_v),T\}$) such that:
\begin{itemize}
	\item[$(i)$] The function $y_\e(\tau,\cdot)\colon U_{\Gamma_\e}\to y_\e\left(\tau,U_{\Gamma_\e}\right)$ is a Lipschitz homeomorphism for all $\tau\in[0,t_v]$.
	\item[$(ii)$] For $t\in[0,t_v]$, let
		$$
		y_{\e,t}^{-1}\colon y_\e(t,U_{\Gamma_\e})\to U_{\Gamma_\e}
		$$
		be the unique function that satisfies $y_{\e,t}^{-1}(y_\e(t,x))=x$ for all $x\in U_{\Gamma_\e}$. The function 
		$$
		y_{\e}^{-1}\colon\bigcup_{t\in [0,t_v]}\big(\{t\}\times y_\e(t,U_{\Gamma_\e})\big)\to U_{\Gamma_\e},\quad y_{\e}^{-1}(t,w):=y_{\e,t}^{-1}(w)
		$$
		is Lipschitz continuous with respect to $t\in [0,t_v]$.
\end{itemize}
\end{lemma}
\begin{proof}
$(i)$. We recall the characterization of $Dy_\e$ established in the proof of the preceding lemma, i.e., \cref{eq:xdery}:
\begin{align*}
Dy_\e(t,x)=\mathds{I}_3+\e\int_0^t\left(A_\e^{(11)}(t,w_\e(\tau,x))Dy_\e(\tau,x)+A_\e^{(12)}(t,w_\e(\tau,x))Dz_\e(\tau,x)\right)\di{\tau}.
\end{align*}
From here, we conclude that
$$
\|Dy_\e(t,\cdot)-\mathds{I}_3\|_\infty\leq 3tl_vC_1(l_v)\quad\text{for all $t\in\overline{S}$}.
$$
This shows (employing the \emph{Neumann} series) that $y_\e(t,\cdot)\colon U_{\Gamma_\e}\to y_\e(t,U_{\Gamma_\e})$ is a Lipschitz homeomorphism for all $t\in[0,t_v]$ where $t_v=\min\{(4l_vC_1(l_v))^{-1},T\}$.
Here, the function $\delta$ is given via $(4l_vC_1(l_v))^{-1}$.
\\[-0.2cm]

$(ii)$. It holds $y_\e(t,y_{\e}^{-1}(t,x))=x$ for all $(t,x)\in\bigcup_{t\in[0,t_v]}\left(\{t\}\times y_\e(t,U_{\Gamma_\e})\right)$.
Implicit differentiation leads to
\begin{align*}
\partial_t\left(y_\e(t,y_{\e}^{-1}(t,x))\right)=\partial_ty_\e(t,y_{\e}^{-1}(t,x))+Dy_\e(t,y_{\e}^{-1}(t,x))\partial_ty_{\e}^{-1}(t,x)=0
\end{align*}
and, therefore, 
\begin{align}\label{eq:Dtyinv}
\partial_ty_{\e}^{-1}(t,x)&=-\left(Dy_\e(t,y_{\e}^{-1}(t,x))\right)^{-1}\partial_ty_\e(t,y_{\e}^{-1}(t,x))\notag\\
&=\e\left(Dy_\e(t,y_{\e}^{-1}(t,x))\right)^{-1}\frac{z_\e(t,y_{\e}^{-1}(t,x))}{|z_\e(t,y_{\e}^{-1}(t,x))|}v_\e(t,y_\e(t,y_{\e}^{-1}(t,x))).
\end{align}
As the right hand side is bounded by virtue of the estimates provided in \Cref{lemma:estimates:yz}, this implies Lipschitz continuity of $y_{\e}^{-1}$ with respect to $t\in[0,t_v]$.
\end{proof}

With the following lemma, we show that any solution of the motion problem given by \cref{eq:level_set_1,eq:level_set_2,eq:level_set_3,eq:level_set_4,eq:level_set_5} can be characterized via $y_\e$ and that, indeed, there is a unique solution to the motion problem.
\begin{lemma}\label{mi:lem:solfbp}
\begin{itemize}
\item[$(i)$] Let $\{\Gamma_\e(t)\}_{t\in[0,t_v]}$ be a solution of the free boundary problem given by \cref{eq:level_set_1,eq:level_set_2,eq:level_set_3,eq:level_set_4,eq:level_set_5}.
Then, for all $t\in[0,t_v]$, $\Gamma_\e(t)=y_\e(t,\Gamma_\e)$.
\item[$(ii)$] There is a unique solution to the motion problem posed in the time interval $[0,t_v]$.
\end{itemize}
\end{lemma}
\begin{proof}
$(i)$. This is shown in \cite[Lemma 3.2]{Ch92} using the method of characteristics.

$(ii)$. This proof follows closely along the lines of \cite[Theorem 3.1]{Ch92} adapting the ideas to our setting.
We introduce a Lipschitz continuous function $\widetilde{\p}_\e\colon [0,t_v]\times \Omega\to[-\e a,\e a]$ via (as a reminder: $\Gamma_\e^{(l)}=\left\{\Lambda_\e(\gamma,l)\ : \ \gamma\in\Gamma_\e\right\}$, see \cref{mi:eq:faminter})
$$
\widetilde{\p}_\e(t,x)=
	\begin{cases} 
	-\e a, \quad &x\in \Omega^{(1)}_\e\setminus y_\e(t,U_{\Gamma_\e})\\ 
	-l	,\quad &x\in y_\e(t,\Gamma_\e^{(l)})\ \text{for some}\ l\in(-\e a,\e a)\\ 
	\e a,\quad &y\in \Omega^{(2)}_\e\setminus y_\e(t,U_{\Gamma_\e})
	\end{cases}.
$$
In the same way as in \cite[Theorem 3.1]{Ch92}, it can be shown that
$$
\nabla\widetilde{\p}_\e(t,x)=z_\e(t,y_\e^{-1}(t,x))\quad\text{for all}\ x\in y_\e(t,U_{\Gamma_\e}),\ t\in[0,t_v]
$$
and, therefore,
\begin{align}\label[in]{phi_inv}
e^{\e l_vt}\geq|\nabla\widetilde{\p}_\e(t,x)|\geq e^{-\e l_vt}\quad\text{for all}\ x\in y_\e(t,U_{\Gamma_\e}),\ t\in[0,t_v]
\end{align}
as well as
\begin{align}\label{eq:motion_problem}
\partial_t\widetilde{\p}_\e(t,y)=\e|\nabla\widetilde{\p}_\e(t,y)|v_\e(t,y)\quad\text{in}\ \bigcup_{t\in[0,t_v]}\left(\{t\}\times y_\e(t,U_{\Gamma_\e})\right).
\end{align}
Due to the Lipschitz continuity of the involved derivatives, we get
$$
\widetilde{\p}_\e\in W^{(2,2),\infty}\left(\bigcup_{t\in[0,t_v]}\left(\{t\}\times y_\e(t,U_{\Gamma_\e})\right)\right).
$$
Now, let $g\colon\R\to[0,1]$ be a $C^2$-function such that $g(0)=0$, $g'(0)=1$, $g'(r)=0$ if $r\notin(-\nicefrac{a}{2},\nicefrac{a}{2})$, and $|g''|\leq\nicefrac{3}{a}$.
We introduce $\p_\e=\e g\circ\left(\e^{-1}\widetilde{\p}_\e\right)\in W^{(2,2),\infty}([0,t_v]\times\Omega)$.
Then, $\p_\e=0$ if and only if $\widetilde{\p}_\e=0$ which implies
$$
\Gamma_\e=\{x\in\Omega \ : \ \p_\e(0,x)=0\}.
$$
and
$$
\{x\in\Omega \ : \ \p_\e(t,x)=0\}=y_\e(t,\Gamma_\e) \quad \text{for all}\ t\in[0,t_v].
$$
It then can easily be checked that $\p_\e$ satisfies the conditions of the motion problem given by \cref{eq:level_set_1,eq:level_set_2,eq:level_set_3,eq:level_set_4,eq:level_set_5}.
\end{proof}

\begin{lemma}\label{lemma:estimate_phi}
There is a continuous function $C_\p\colon[0,\infty)\to(0,\infty)$ such that
\begin{align*}
\e^{-1}\|\partial_t\widetilde{\p}_\e\|_\infty+\|\partial_t\nabla\widetilde{\p}_\e\|_\infty&\leq l_vC_\p(l_v),\\
\|\nabla\widetilde{\p}_\e\|_\infty+\e\|D^2\widetilde{\p}_\e\|_\infty&\leq C_\p(l_v).
\end{align*}
\end{lemma}
\begin{proof}
In this proof, we rely on the estimates provided in \Cref{lemma:estimates:yz}.
Let $t\in[0,t_v]$ and $x\in y_\e(t,U_{\Gamma_\e})$.	
The second spatial derivative is given as
$$
D^2\widetilde{\p}_\e(t,x)=(Dy_\e(t,y_\e^{-1}(t,x)))^{-1}Dz_\e(t,y_\e^{-1}(t,x))
$$
and can therefore be estimated via
$$
\left|D^2\widetilde{\p}_\e(t,x)\right|\leq\frac{4}{\e}C_w(l_v)
$$
where $C_w$ is the function given by \Cref{lemma:estimates:yz}.
Furthermore, as $\widetilde{\p}_\e$ satisfies \cref{phi_inv} and \cref{eq:motion_problem}, we can estimate
$$
|\partial_t\widetilde{\p}_\e(t,x)|\leq\e|\nabla\widetilde{\p}_\e(t,x)||v_\e(t,x)|\leq \e e^{\e l_vt}l_v.
$$
Taking the derivative with respect to $x\in y_\e(t,U_{\Gamma_\e})$ in \cref{eq:motion_problem}, we get
$$
\partial_t\nabla\widetilde{\p}_\e(t,x)=\e|\nabla\widetilde{\p}_\e(t,x)|\nabla v_\e(t,x)+\e D^2\widetilde{\p}_\e(t,x)\frac{\nabla\widetilde{\p}_\e(t,x)}{|\nabla\widetilde{\p}_\e(t,x)|}|v_\e(t,x)|
$$
and find the upper bound
$$
|\partial_t\nabla\widetilde{\p}_\e(t,x)|\leq l_v\left(\e_0 e^{\e_0 l_vt_v}+4C_w(l_v)\right).
$$
\end{proof}

\subsection{Motion function}\label{mi:ssec:motion}
For $\e>0$ and $\gamma\in\Gamma_\e$, we introduce the function $F_{\e,\gamma}\colon[0,t_v]\times(-\e a,\e a)\to\R$ via $F_{\e,\gamma}(t,r)=\p_\e(t,\Lambda_\e(\gamma,r))$.
Then, $F_{\e,\gamma}(0,0)=\p_\e(0,\Lambda_\e(\gamma,0))=0$ for all $\gamma\in\Gamma_\e$.

\begin{lemma}\label{lemma_estimate_Fegamma}
For all $\e>0$ and $\gamma\in\Gamma_\e$, it holds $\partial_2F_{\e,\gamma}(0,0)=-1$.
Furthermore, there are $\widetilde{t}_v\in[0,t_v]$ and $0<R_v<a$ such that $\partial_2F_{\e,\gamma}(t,r)\leq-\nicefrac{1}{3}$ for all $t\in[0,\widetilde{t}_v]$ and $r\in[-\e R_v,\e R_v]$.
\end{lemma}
\begin{proof}
We calculate
\begin{align}\label{lemma_estimate_Fegamma_1}
\partial_2F_{\e,\gamma}(t,r)
&=g'(\e^{-1}\widetilde{\p}_\e(t,\Lambda_\e(\gamma,r)))\nabla\widetilde{\p}_\e(t,\Lambda_\e(\gamma,r))\cdot n_{\Gamma_\e}(\gamma)
\end{align}
and see that
$$
\partial_2F_{\e,\gamma}(0,0)=-1<0.
$$
For any $t\in[0,t_v]$ and $r\in(-\e a,\e a)$, we have
$$
\partial_2F_{\e,\gamma}(t,r)=-1+\int_0^r\partial_2^2F_{\e,\gamma}(0,s)\di{s}+\int_0^t\partial_t\partial_2F_{\e,\gamma}(\tau,r)\di{\tau}.
$$
Starting off with the first integrand, $\partial_2^2F_{\e,\gamma}$, we get
\begin{multline*}
\partial_2^2F_{\e,\gamma}(t,r)=\e^{-1}g''(\e^{-1}\widetilde{\p}_\e(t,\Lambda_\e(\gamma,r)))\left(\nabla\widetilde{\p}_\e(t,\Lambda_\e(\gamma,r))\cdot n_{\Gamma_\e}(\gamma)\right)^2\\
+D^2\widetilde{\p}_\e(t,\Lambda_\e(\gamma,r))n_{\Gamma_\e}(\gamma)\cdot n_{\Gamma_\e}(\gamma).
\end{multline*}
Using the estimates collected in \Cref{lemma:estimate_phi}, we can conclude that
$$
\e\left|\partial_2^2F_{\e,\gamma}(t,r)\right|\leq\frac{3}{a}e^{2\e l_vt}+C_\p(l_v).
$$
For the second integrand, $\partial_t\partial_2F_{\e,\gamma}$, we calculate
\begin{multline*}
\partial_t \partial_2F_{\e,\gamma}(t,r)
=\e^{-1}g''(\e^{-1}\widetilde{\p}_\e(t,\Lambda_\e(\gamma,r)))\partial_t\widetilde{\p}_\e(t,\Lambda_\e(\gamma,r))\nabla\widetilde{\p}_\e(t,\Lambda_\e(\gamma,r))\cdot n_{\Gamma_\e}(\gamma)\\
+g'(\e^{-1}\widetilde{\p}_\e(t,\Lambda_\e(\gamma,r)))\partial_t\nabla\widetilde{\p}_\e(t,\Lambda_\e(\gamma,r))\cdot n_{\Gamma_\e}(\gamma)
\end{multline*}
and estimate
\begin{align*}
\left|\partial_t \partial_2F_{\e,\gamma}(t,r)\right|\leq \frac{3}{a}l_vC_\p(l_v)\left(C_\p(l_v)+1\right)
\end{align*}
and finally arrive at
$$
\partial_2F_{\e,\gamma}(t,r)\leq-1+\frac{r}{\e}\left(\frac{3}{a}e^{2\e_0l_vt}+C_\p(l_v)\right)+tl_v\left(\frac{3}{a}C_\p(l_v)\left(C_\p(l_v)+1\right)\right).
$$
\end{proof}

\begin{theorem}[Height function]\label{mi:thm:height}
There is a time $T_v\in(0,\tau_v]$ monotonically decreasing with respect to $l_v$ and such that $T_v=T$ for $l_v$ sufficiently small such that:
\begin{itemize}
\item[$(i)$]
There is a height function $h_\e\colon\Gamma_\e\times[0,T_v]\to(-\e a,\e a)$ satisfying
$$
\Gamma_\e(t)=\{\Lambda_\e(\gamma, h_\e(t,\gamma)) \ : \ \gamma \in\Gamma_\e\} \quad\text{for all}\ t\in [0,T_v]
$$
\item[$(ii)$] It holds the estimate
$$
\frac{5}{\e a}\|h_\e\|_{L^\infty((0,T_v)\times\Gamma_\e)}+2\|\nabla_{\Gamma_\e} h_\e\|_{L^\infty((0,T_v)\times\Gamma_\e)}\leq\frac{1}{2}.
$$
Moreover, $\|\partial_th_\e\|_\infty\leq3\e l_vC_\p(l_v)$.
\end{itemize}
\end{theorem}
\begin{proof}
$(i)$.
Note that $F_{\e,\gamma}(0,0)=0$ and $\partial_2F_{\e,\gamma}(0,0)=-1$.
By the \emph{Implicit Function Theorem}, we infer that, for every $\e>0$ and for every $\gamma\in\Gamma_\e$, there is a time $\tau_{\e,\gamma}>0$ and a differentiable function $h_{\e,\gamma}\colon[0,\tau_{\e,\gamma}]\to(-\e a,\e a)$ such that $F_{\e,\gamma}(t,h_{\e,\gamma}(t))=0$ for all $t\in[0,\tau_{\e,\gamma}]$.
Let $\tau_{\e,\gamma}\in\overline{S}$ always be the maximal possible point in time for this to be true.
%
It holds that
$$
\sup\{|h_{\e,\gamma}(t)|\ :\ \gamma\in\Gamma_\e\}=\sup\{|d_{\Gamma_\e}(y_\e(t))|\ : \ \gamma\in\Gamma_\e \}\leq\e tl_v\quad\text{for all}\ t\in[0,\tau_{\e,\gamma}],
$$
%
Here, the equality holds due to
$$
\Gamma_\e(t)=\{\Lambda_\e(\gamma, h_{\e,\gamma}(t)) \ : \ \gamma \in\Gamma_\e\}
=\{y_\e(t,\gamma) \ : \ \gamma \in\Gamma_\e\}.
$$
And, for the inequality, we observe that $y_\e(0,\gamma)\in\Gamma_\e$ and that $y_\e$ satisfies \cref{s:ode:1}.
Now, take $\tau_v=\min\{t_v,l_v^{-1}R_v\}$.
We claim that
$$
\inf\{\tau_{\e,\gamma}\ : \e>0,\ \gamma\in\Gamma_\e\}\geq \tau_v.
$$
Let us assume this is not the case, i.e., there are $\e>0$ and $\gamma\in\Gamma_\e$ such that $\tau_{\e,\gamma}<\tau_v$. 
Since 
\begin{align*}
&(i)\ F_{\e,\gamma}(\tau_{\e,\gamma},h_{\e,\gamma}(\tau_{\e,\gamma}))=0,\\
&(ii)\ \partial_2F_{\e,\gamma}(\tau_{\e,\gamma},h_{\e,\gamma}(\tau_{\e,\gamma}))<-\frac{1}{3},
\end{align*}
we can apply the \emph{Implicit Function Theorem} again which contradicts the assumption that $\tau_{\e,\gamma}$ is maximal.
Here, $(ii)$ holds true by virtue of \Cref{lemma_estimate_Fegamma}.
As a consequence, we are able to define $h_\e\colon[0,\tau_v]\times\Gamma_\e\to(-\e a,\e a)$ via $h_\e(t,\gamma):=h_{\e,\gamma}(t)$.
\\[-0.2cm]

$(ii)$. Owing to the regularity of $\Lambda_\e$ and $\p_\e$, we have $h_\e\in W^{2,\infty}((0,T)\times\Gamma_\e)$.
For all $t\in[0,\tau_v]$ and $\gamma\in\Gamma_\e$, we have $F_{\e,\gamma}(t,h_\e(t,\gamma))=0$ implying vanishing derivatives with respect to time and space.
Implicit differentiation with respect to time yields
\begin{align}\label{eq:timeh_e}
\partial_th_\e(t,\gamma)=-\frac{\partial_tF_{\e,\gamma}(t,h_\e(t,\gamma))}{\partial_2F_{\e,\gamma}(t,h_\e(t,\gamma))}.
\end{align}
Considering that $\|g'\|_\infty\leq1$, we are therefore led to
$$
|\partial_th_\e(t,\gamma)|\leq3\left|\partial_t\widetilde{\p}_\e(t,\Lambda_\e(\gamma,h_\e(t,\gamma)))\right|\leq3\e l_vC_\p(l_v).
$$
Let us first observe that $\nabla_{\Gamma_\e}h_\e(t,\gamma)=0$ if and only if
$$
n_{\Gamma_\e}(t,\Lambda_\e(\gamma,h_\e(t,\gamma))=n_{\Gamma_\e}(\gamma).
$$ 
The normal vector at $\gamma\in\Gamma_\e(t)$ is given as
$$
n_{\Gamma_\e}(t,\gamma)=\frac{\nabla\p_\e(t,\gamma)}{|\nabla\p_\e(t,\gamma)|}=\frac{\nabla\widetilde{\p}_\e(t,\gamma)}{|\nabla\widetilde{\p}_\e(t,\gamma)|}.
$$
For the surface gradient of $h_\e$, we can find the representation (we point to \cite[Section 2.5]{PS16})
\begin{equation}\label{eq:surface_gradient_h}
\nabla_{\Gamma_\e} h_\e(t,\gamma)=\left(\mathds{I}_3-h_\e(t,\gamma)L_{\Gamma_\e}(\gamma)\right)\left(n_{\Gamma_\e}(\gamma)-\frac{1}{n_{\Gamma_\e}(t,\overline{\gamma}_t)\cdot n_{\Gamma_\e}(\gamma)}n_{\Gamma_\e}(t,\overline{\gamma}_t)\right),
\end{equation}
where we have set $\overline{\gamma}_t=y_\e(t,\gamma)$.
Due to
$$
n_{\Gamma_\e}(t,\overline{\gamma}_t)=n_{\Gamma_\e}(\gamma)+\int_0^t\underbrace{\frac{\partial_t\nabla\widetilde{\p}_\e(t,\overline{\gamma}_t)|\nabla\widetilde{\p}_\e(t,\overline{\gamma}_t)|-\nabla\widetilde{\p}_\e(t,\overline{\gamma}_t)\partial_t|\nabla\widetilde{\p}_\e(t,\overline{\gamma}_t)|}{|\nabla\widetilde{\p}_\e(t,\overline{\gamma}_t)|^2}}_{=:\Phi_\e(\tau,\overline{\gamma}_t)}\di{\tau},
$$
we estimate
$$
|n_{\Gamma_\e}(\gamma)-n_{\Gamma_\e}(t,\overline{\gamma}_t)|\leq\int_0^t|\Phi_\e(\tau,\overline{\gamma}_t)|\di{\tau}\leq 2tl_ve^{3\e_0 l_vt}C_\p(l_v),
$$
and (for $t$ small enough, but independent of $\e$ and decreasing with increasing $l_v$)
$$
0<1-2te^{3\e l_vt}l_vC_\p(l_v)\leq n_{\Gamma_\e}(\gamma)\cdot n_{\Gamma_\e}(t,\overline{\gamma}_t)\leq1.
$$
Combining these estimates to bound the difference
\begin{multline*}
n_{\Gamma_\e}(\gamma)-\frac{1}{n_{\Gamma_\e}(t,\overline{\gamma}_t)\cdot n_{\Gamma_\e}(\gamma)}n_{\Gamma_\e}(t,\overline{\gamma}_t)\\
=n_{\Gamma_\e}(\gamma)-n_{\Gamma_\e}(t,\overline{\gamma}_t)+\frac{n_{\Gamma_\e}(t,\overline{\gamma}_t)\cdot n_{\Gamma_\e}(\gamma)-1}{n_{\Gamma_\e}(t,\overline{\gamma}_t)\cdot n_{\Gamma_\e}(\gamma)}n_{\Gamma_\e}(t,\overline{\gamma}_t),
\end{multline*}
we are led to
$$
\left|n_{\Gamma_\e}(\gamma)-\frac{1}{n_{\Gamma_\e}(t,\overline{\gamma}_t)\cdot n_{\Gamma_\e}(\gamma)}n_{\Gamma_\e}(t,\overline{\gamma}_t)\right|\leq2tl_ve^{3\e_0 l_v}C_\p(l_v)\left(1+\sum_{k=0}^\infty\left(2tl_ve^{3\e_0l_vt}C_\p(l_v)\right)^k\right).
$$
In summary, estimating \cref{eq:surface_gradient_h} leads us to
$$
|\nabla_{\Gamma_\e} h_\e(t,\gamma)|\leq\left(1+\frac{tl_v}{2a}\right)\left(2te^{3\e_0 l_vt}l_vC_\p(l_v)\left(1+\sum_{k=0}^\infty\left(2tl_ve^{3\e_0 l_vt}C_\p(l_v)\right)^k\right)\right).
$$
\end{proof}

Let $\chi\in\mathcal{D}(\R_\geq0)$ be a cut-off function that satisfies
$$
0\leq\chi\leq1,\qquad \chi(r)=1\ \text{if}\ r<\frac{1}{3},\qquad\chi(r)=0\ \text{if}\ r>\frac{2}{3}.
$$
In addition, let $\chi'(r)<0$ if $1/3<r<2/3$ as well as $\|\chi'\|_\infty\leq4$.

We introduce the function $s_\e\colon[0,T_v]\times\overline{\Omega}\to\overline{\Omega}$ via
\begin{align}\label{def_se}
s_\e(t,x)=
	\begin{cases}
	x+h_\e(t,P_{\Gamma_\e}(x))n_{\Gamma_\e}(P_{\Gamma_\e}(x))\chi\left(\frac{\dist(x,\Gamma_\e)}{\e a}\right),\quad &x\in U_{\Gamma_\e}\\
	x ,\quad &x\notin U_{\Gamma_\e}\end{cases}.
\end{align}

\begin{lemma}\label{mi:lem:psi}
The function $s_\e\colon[0,T_v]\times\overline{\Omega}\to\overline{\Omega}$ is a regular $C^1$-motion with $\Gamma_\e(t)=s_\e(t,\Gamma_\e)$ for all $t\in[0,T_v]$.
\end{lemma}
\begin{proof}
With the estimates provided in \Cref{mi:thm:height}, we can conclude that $s_\e(t,\cdot)\colon\overline{\Omega}\to\overline{\Omega}$ is a regular $C^1$-deformation with $\Gamma_\e(t)=s_\e(t,\Gamma_\e)$ for all $t\in[0,T_v]$.
For details, we refer to \cite[Lemma 2.9]{E18}.
The regularity with respect to time follows via $h_\e\in C^{1,1}([0,T_v]\times\overline{\Omega})$.
\end{proof}
\section{Limit behavior (proof of \texorpdfstring{\Cref{theorem:3}}{Theorem 6.2})}\label{mi:sec:limit}
In this section, the limit behavior of the functions related to the Hanzawa transformation $s_\e$ as given by \Cref{mi:lem:psi}, in particular $F_\e=Ds_\e$ and $J_\e=\det F_\e$, are investigated.
To be able to pass to the limit $\e\to0$, strong two-scale convergence of these quantities has to be established.
We start by introducing the folding and unfolding operators, similar (in spirit) considerations can be found, e.g., in~\cite{MP08}, and by formulating a few technical lemmas. 

In an effort to keep the notations for the estimations shorter, we introduce functions
\begin{subequations}
\begin{alignat}{2}
q_\e&\colon S_v\times U_{\Gamma_\e}\to\R^3,&\quad q_\e(t,x)&:=z_\e(t,y_\e^{-1}(t,x)),\label{seq:composita_q} \\
\eta_\e&\colon S_v\times\Gamma_\e\to\Omega,&\quad \eta_\e(t,\gamma)&:=\Lambda_\e(\gamma,h_\e(t,\gamma)).\label{seq:composita_eta}
\end{alignat}
\end{subequations}
\subsection{Preliminaries and auxiliary lemmas}
\begin{figure}[!h]
\centering
\begin{tikzpicture}[scale=1.7]
\begin{axis}[
	unit vector ratio*=1 1 1,
  axis lines=middle,
  axis line style={thick},
  xmin=-0.5,xmax=4.5,ymin=-0.5,ymax=2.5,
	grid=major,
  grid style={densely dotted,black!80}]
\draw[thick, fill=black!15] (0,0) rectangle (1,1);
\draw[fill] (3.2,1.8) circle (0.03);
\draw[fill] (3,1) circle (0.03);
\draw[fill] (0.2,0.8) circle (0.03);
\draw (3.3,1.8) node[above]{{\footnotesize{$x$}}};
\draw (3.2,1.05) node[below]{{\footnotesize{$[x]$}}};
\draw (0.45,0.85) node[below]{{\footnotesize{$\{x\}$}}};
\draw[thin,->] (0,0) -- (3.18,1.79);
\draw[thin,->] (0,0) -- (3,1);
\draw[thick, dotted] (3,1) -- (3.2,1.8);
\draw[thin,->] (0,0) -- (0.2,0.8);
\end{axis}
\end{tikzpicture}
\caption{Simple example demonstrating the construction of $[x]$ and $\{x\}$.}
\label{pr:fig:unfolding}
\end{figure}
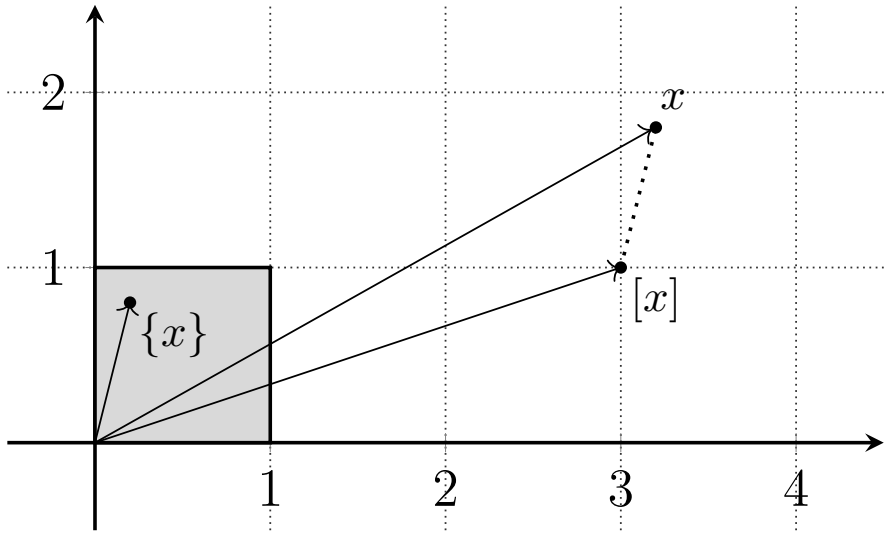

For $x\in\R^3$, $[x]$ is defined to be the unique $k\in\Z^3$ such that $\{x\}:=x-[x]\in[0,1)^3$ and, for functions $f\colon\Omega\to\R$ and $f_b\colon\Gamma_\e\to\R$, we denote the periodic unfolding via $\decon{f}\colon\Omega\times Y\to\R$ and $\decon{f_b}\colon\Omega\times\Gamma\to\R$ defined by
$$
\decon{f}(x,y)=f\left(\e y+\left[\frac{x}{\e}\right]\right),\quad \decon{f_b}(x,\gamma)=f_b\left(\e\gamma+\left[\frac{x}{\e}\right]\right).
$$
We get the integral identities (see \cite{CDG02})
\begin{align*}
\int_{\Omega}f(x)\di{x}=\int_{\Omega\times Y}\decon{f}(x,y)\di{(x,y)},\\
\int_{\Gamma_\e}f_b(x)\di{x}=\frac{1}{\e}\int_{\Omega\times\Gamma}\decon{f_b}(x,y)\di{(x,y)}
\end{align*}
and, for $\id\colon\Omega\to\Omega$ and $n,m\in\N$, it holds
\begin{align}\label[in]{eq:est_unfold}
\left|\decongen{\id}{n}-\decongen{\id}{m}\right|\leq\sqrt{2}\left(\e_n+\e_m\right).
\end{align}
In addition, for functions $g\colon\Omega\times Y\to\R$ and $g_b\colon\Omega\times\Gamma\to\R$, we set
\begin{alignat*}{2}
\recon{g}&\colon\Omega\to\R,&\quad \recon{g}(x)&=g\left(x,\deval{x}\right),\\
\recon{g_b}&\colon\Gamma_\e\to\R,&\quad \recon{g_b}(x)&=g_b\left(x,\deval{x}\right).
\end{alignat*}
We find that, $f\in W^{1,2}(\Omega;W^{1,2}_\#(Y))$,
\begin{equation}\label{in:est_foldunfold}
\big\|f-\decon{\recon{f}}\big\|_{L^2(\Omega\times Y)}^2
	\rightarrow0
\end{equation}
as $\left(\e y+\e\eeval{x},\left[y+\eeval{x}\right]\right)$ converges uniformly to $(x,y)$. 

The following identities are a consequence of the periodicity of the initial configuration.
For $x\in U_{\Gamma_\e}$, $y\in Y$, $\gamma\in\Gamma$, and $r\in(-\e a,\e a)$, it holds
\begin{subequations}\label{eq:decongeometric}
\begin{align}
\decon{n_\e}(x,\gamma)&=n(\gamma),\label{seq:decongeometric:1}\\
\decon{\Lambda_\e}(x,\gamma,r)&=\e\Lambda\left(\gamma,\frac{r}{\e}\right)+\e\eeval{x},\label{seq:decongeometric:4}\\
\decon{L_{\Gamma_\e}}(x,\gamma)&=\e^{-1}L_\Gamma(\gamma),\label{seq:decongeometric:5}\\
\decon{P_{\Gamma_\e}}(x,y)&=\e P_{\Gamma}(y)+\e\eeval{x},\label{seq:decongeometric:3}\\
\decon{DP_{\Gamma_\e}}(x,y)&=\big(\mathds{I}-d_\Gamma(y)L_\Gamma(P_\Gamma(y))\big)^{-1}\left(\mathds{I}-n(P_\Gamma(y))\otimes n(P_\Gamma(y))\right)\label{seq:decongeometric:6}.
\end{align}
\end{subequations}
With these relations in mind, we are able to connect the limit behavior of the auxiliary function $\eta_\e$ and the height function $h_\e$.

\begin{lemma}\label{lemma:etatomu}
Let $n$, $m\in\N$.
It holds
\begin{align*}
\big|\e_n^{-1}\decongen{\eta_{\e_n}}{n}-\e_m^{-1}\decongen{\eta_{\e_m}}{m}\big|
&\leq\big|\e_n^{-1}\decongen{h_{\e_n}}{n}-\e_m^{-1}\decongen{h_{\e_m}}{m}\big|
\end{align*}
as well as
\begin{equation*}
\big|\decongen{D\eta_{\e_n}}{n}-\decongen{D\eta_{\e_m}}{m}\big|
\leq\frac{1}{2a}\big|\e_n^{-1}\decongen{h_{\e_n}}{n}-\e_m^{-1}\decongen{h_{\e_n}}{n}\big|+\big|\decongen{\nabla h_{\e_n}}{n}-\decongen{\nabla h_{\e_m}}{m}\big|.
\end{equation*}
\end{lemma}
\begin{proof}
Since $\Lambda$ is contractive and \cref{seq:decongeometric:1,seq:decongeometric:4} hold, we conclude
\begin{multline*}
\big|\e_n^{-1}\decongen{\eta_{\e_n}}{n}-\e_m^{-1}\decongen{\eta_{\e_m}}{m}\big|\\
=\big|\Lambda\left(\gamma,\e_n^{-1}\decongen{h_{\e_n}}{n}\right)-\Lambda(\gamma,\e_m^{-1}\decongen{h_{\e_m}}{m})\big|
\leq\big|\e_n^{-1}\decongen{h_{\e_n}}{n}-\e_m^{-1}\decongen{h_{\e_m}}{m}\big|.
\end{multline*}
The spatial derivative of $\eta_\e$ is given as
$$
D_{\Gamma_\e}\eta_\e=\mathrm{Id}+\nabla_{\Gamma_\e}h_\e\otimes n_\e-h_\e L_{\Gamma_\e}.
$$
Using \cref{seq:decongeometric:1,seq:decongeometric:4,seq:decongeometric:5}, we estimate
\begin{equation*}
\big|\decongen{D_{\Gamma_\e}\eta_{\e_n}}{n}-\decongen{D_{\Gamma_\e}\eta_{\e_m}}{m}\big|
	\leq\left|\decongen{\nabla_{\Gamma_{\e_n}}h_{\e_n}}{n}-\decongen{\nabla_{\Gamma_{\e_m}}h_{\e_m}}{m}\right|
	+\frac{1}{2a}\big|\e_n^{-1}\decongen{h_{\e_n}}{n}-\e_m^{-1}\decongen{h_{\e_n}}{n}\big|.
\end{equation*}
\end{proof}

In the next few lemmas, we establish some technical results which are needed to show the strong two-scale convergence of $F_\e$ and $J_\e$.

\begin{lemma}\label{lemma:trace_convergence}
\begin{itemize}
	\item[$(i)$] Let $u_\e\in W^{1,2}(\Omega)$ and $u\in L^2(\Omega;W^{1,2}_\#(Y))$ such that $\decon{u_\e}\to u$ and $\e\decon{\nabla u_\e}\to\nabla_y u$ strongly in $L^2(\Omega\times Y)$.
Then, $\decon{u_\e}\to u$ strongly in $L^2(\Omega\times\Gamma)$.
	\item[$(ii)$] For all $u\in W^{1,2}(\Omega)$, it holds that
$$
\e\|u\|^2_{L^2(\Gamma_\e(t))}\leq 4C_{tr}\left(\|u\|^2_{L^2(\Omega)}+\e^2\|\nabla u\|^2_{L^2(\Omega)}\right).
$$
\end{itemize}
\end{lemma}
\begin{proof}
$(i)$. This is due to the trace embedding operator $W^{1,2}(Y)\hookrightarrow L^2(\Gamma)$.

$(ii)$. Let $C_{tr}$ be the trace constant of the embedding $W^{1,2}(\Omega)\hookrightarrow L^2(\Gamma_\e)$.
For $u\in W^{1,2}(\Omega)$ and $t\in[0,T_v]$, we have
\begin{align*}
\e\int_{\Gamma_\e(t)}|u(\gamma)|^2\di{\gamma}
	&=\e\int_{\Gamma_\e}|u(y_\e(t,\gamma))|^2|\det(D_{\Gamma_\e}y_\e(t,\gamma))|\di{\gamma}\\
	&\leq2C_{tr}\left(\int_{\Omega}|u\circ y_\e(x)|^2\di{x}+\e^2\int_{\Omega}|\nabla(u\circ y_\e)(x)|^2\di{x}\right)\\
	&\leq2C_{tr}\left(\int_{\Omega}|u\circ y_\e(x)|^2\di{x}+2\e^2\int_{\Omega}|\nabla u\circ y_\e(x)|^2\di{x}\right).
\end{align*}
The time parametrized coordinate transformation $x\mapsto y_\e^{-1}(t,x)$ (note that $y_\e^{-1}(t,\Omega)=\Omega$) then leads to
$$
\e\|u\|_{L^2(\Gamma_\e(t))}^2\leq4C_{tr}\left(\|u\|^2_{L^2(\Omega)}+\e^2\|\nabla u\|^2_{L^2(\Omega)}\right)
$$
\end{proof}

Parts of the analysis rely on the ability to estimate certain differences of some composites of functions involving $y_\e$.
In the following lemma, we collect some general results.
\begin{lemma}\label{lemma:decon_estimate}
Let $(f_\e)\subset W^{1,\infty}(\Omega)$ and $n,m\in\N$ ($n>m$).
\begin{enumerate}
\item\label{it:lemma:decon_estimate:1} Let $\|\nabla f_{\e_m}\|_\infty$ be bounded independently of the parameter $\e$ and $\decon{f_\e}$ be a Cauchy sequence.
Then, there are $C,C_m>0$ such that 
\begin{equation*}
\|f_{\e_n}(\decongen{y_{\e_n}}{n})-f_{\e_m}(\decongen{y_{\e_m}}{m})\|_{L^2(\Omega\times Y)}
	\leq C_m
	+C\big\|\decongen{y_{\e_n}}{n}-\decongen{y_{\e_m}}{m}\big\|_{L^2(\Omega\times Y)}^2.
\end{equation*}
and such that $\lim_{m\to\infty}C_m=0$.
\item\label{it:lemma:decon_estimate:2} Let $f\in W^{1,\infty}(\Omega;W^{1,\infty}_\#(Y))$ such that $\decon{f_\e}\to f$.
For $g_\e=y_\e$ or $g_\e=y_\e^{-1}$, we can estimate
\begin{multline*}
\|f_{\e_n}(\decongen{g_{\e_n}}{n})-f_{\e_m}(\decongen{g_{\e_m}}{m})\|^2_{L^2(\Omega\times Y)}\\
	\leq C_m+C\left(\big\|\decongen{g_{\e_n}}{n}-\decongen{g_{\e_m}}{m}\big\|^2_{L^2(\Omega\times Y)}
	+\big\|\e_n^{-1}\decongen{g_{\e_n}}{n}-\e_m^{-1}\decongen{g_{\e_m}}{m}\big\|^2_{L^2(\Omega\times Y)}\right)
\end{multline*}
where $C,C_m>0$ and $\lim_{m\to\infty}C_m=0$.
\item\label{it:lemma:decon_estimate:3} Let $f\in W^{1,\infty}(\Omega;W^{1,\infty}_\#(Y))$ such that $\decon{f_\e}\to f$ and $\e\decon{\nabla f_\e}\to\nabla_yf$.
Then, we estimate
\begin{multline*}
\|f_{\e_n}(\decongen{\eta_{\e_n}}{n})-f_{\e_m}(\decongen{\eta_{\e_m}}{m})\|^2_{L^2(\Omega\times\Gamma)}\\
	\leq C_m+C\left(\big\|\decongen{h_{\e_n}}{n}-\decongen{h_{\e_m}}{m}\big\|^2_{L^2(\Omega\times\Gamma)}
	+\big\|\e_n^{-1}\decongen{h_{\e_n}}{n}-\e_m^{-1}\decongen{h_{\e_m}}{m}\big\|^2_{L^2(\Omega\times\Gamma)}\right)
\end{multline*}
where $C,C_m>0$ and $\lim_{m\to\infty}C_m=0$.
\end{enumerate}
\end{lemma}
\begin{proof} 
Proofs of these technical estimates are given in \cite[Lemma 6.20]{E18}.
\end{proof}

\subsection{Limit behavior}
Based on the estimates established via \Cref{lemma:estimates:yz}, it is clear that $y_\e$ converges strongly to the identity operator and that both $Dy_\e$ and $z_\e$ have two-scale converging subsequences.
This in itself, however, is not enough to guarantee strong convergence of their unfolded counterparts, which in consequence may also impede strong convergence of $\decon{F_\e}$ and $\decon{J_\e}$ -- a property that is needed to make sure that passing to the limit $\e\to0$ is justified.
%

In the following lemma, we investigate the limit behavior of the dilated functions $\e^{-1}\decon{y_\e-\mathrm{Id}}$ and $\decon{z_\e}$.
\begin{lemma}\label{lemma:twosC_yz}
There exist functions $y,z\in L^2(S\times\Omega;H^1_\#(Y))^3$ such that
$$
\frac{1}{\e}\decon{y_\e-\mathrm{Id}}\to y-\mathrm{Id},\quad \decon{z_\e}\to z,\quad \decon{Dy_\e}\to D_yy,\quad \e\decon{Dz_\e}\to D_yz.
$$
\end{lemma}
\begin{proof}
Let $\delta>0$ be given and let $n,m\in\N$, such that $n>m$ and such that $e^{\e_ml_vT_v}<2$.\footnote{This is a mere technicality to allow for a more compact notation of the estimates. Here, we do not care about the details of the specific estimates, we only want to ensure convergence.}
Taking a look at the ODE sytem given by \cref{s:ode:1,s:ode:2,s:ode:3,s:ode:4} and its corresponding system that emerges by differentiation with respect to the spatial variable, we find that (in $S\times\Omega\times\Sigma$, $(i=n,m)$)
\begin{subequations}
\begin{align}
\e_i^{-1}\partial_t\decongen{y_{\e_i}-\mathrm{Id}}{i}&=\frac{\decongen{z_{\e_i}}{i}}{|\decongen{z_{\e_i}}{i}|}v_{\e_i}(\decongen{y_{\e_i}}{i}),\\
\partial_t\decongen{z_{\e_i}}{i}&=\e_i\big|\decongen{z_{\e_i}}{i}\big|\nabla v_{\e_i}(\decongen{y_{\e_i}}{i}),\\[0.3cm]
\partial_t\decongen{Dy_{\e_i}}{i}&=\e_iA_{\e_i}^{(11)}\left(\decongen{w_{\e_i}}{i}\right)\decongen{Dy_{\e_i}}{i}
	+\e_iA_{\e_i}^{(11)}\left(\decongen{w_{\e_i}}{i}\right)\decongen{Dz_{\e_i}}{i},\\
\e_i\partial_t\decongen{Dz_{\e_i}}{i}&=\e_i^2A_{\e_i}^{(21)}\left(\decongen{w_{\e_i}}{i}\right)\decongen{Dy_{\e_i}}{i}
	+\e_i^2A_{\e_i}^{(22)}\left(\decongen{w_{\e_i}}{i}\right)\decongen{Dz_{\e_i}}{i}.
\end{align}
\end{subequations}
Now, subtracting these equations for $i=n$ and $i=m$ from one another, multiplying with the corresponding differences, and integrating over $\Omega\times Y$, we are led to
\begin{subequations}\label[in]{ineq:st2sc}
\begin{multline}\label[in]{ineq:st2sC_y}
\ddt\big\|\e_n^{-1}\decongen{y_{\e_n}-\mathrm{Id}}{n}-\e_m^{-1}\decongen{y_{\e_m}-\mathrm{Id}}{m}\big\|^2_{L^2(\Omega\times Y)}\\
	\leq2\int_{\Omega\times Y}\Big|\frac{\decongen{z_{\e_n}}{n}}{|\decongen{z_{\e_n}}{n}|}v_{\e_n}(\decongen{y_{\e_n}}{n})-\frac{\decongen{z_{\e_m}}{m}}{|\decongen{z_{\e_m}}{m}|}v_{\e_m}(\decongen{y_{\e_m}}{m})\Big|\\
	\Big|\e_n^{-1}\decongen{y_{\e_n}-\mathrm{Id}}{n}-\e_m^{-1}\decongen{y_{\e_m}-\mathrm{Id}}{m}\Big|\di{(x,y)},
\end{multline}
\begin{multline}\label[in]{ineq:st2sc_z}
\ddt\big\|\decongen{z_{\e_n}}{n}-\decongen{z_{\e_m}}{m}\big\|^2_{L^2(\Omega\times Y)}\\
	\leq2\int_{\Omega\times Y}\Big|\e_n\big|\decongen{z_{\e_n}}{n}\big|\nabla v_{\e_n}(\decongen{y_{\e_n}}{n})-\e_m\big|\decongen{z_{\e_m}}{m}\big|\nabla v_{\e_m}(\decongen{y_{\e_m}}{m})\Big|\\
	\Big|\decongen{z_{\e_n}}{n}-\decongen{z_{\e_m}}{m}\Big|\di{(x,y)}.
\end{multline}
%
\end{subequations}
To proceed in showing that these sequences are Cauchy sequences, several independent estimates are needed to manage the right hand sides of \cref{ineq:st2sC_y,ineq:st2sc_z}.
In the following, we heavily rely on the estimates established by \Cref{lemma:estimates:yz}.
With the reverse triangle inequality, we get
\begin{subequations}
\begin{equation}\label[in]{ineq:st2sc:1}
\big||\decongen{z_{\e_n}}{n}|-|\decongen{z_{\e_m}}{m}|\big|\leq\big|\decongen{z_{\e_n}}{n}-\decongen{z_{\e_m}}{m}\big|,
\end{equation}
Since $e^{\e_ml_vT_v}<2$, we also see that
\begin{equation}\label[in]{ineq:st2sc:2}
\left|\frac{\decongen{z_{\e_n}}{n}}{|\decongen{z_{\e_n}}{n}|}-\frac{\decongen{z_{\e_m}}{m}}{|\decongen{z_{\e_m}}{m}|}\right|\leq10\left|\decongen{z_{\e_n}}{n}-\decongen{z_{\e_m}}{m}\right|.
\end{equation}
Moreover, for $f_\e=v_{\e},\e\nabla v_{\e}$, we can apply \Cref{lemma:decon_estimate} to get
\begin{multline}\label[in]{ineq:st2sc:3}
\big\|f_{\e_n}(\decongen{y_{\e_n}}{n})-f_{\e_m}(\decongen{y_{\e_n}}{m})\big\|_{L^2(\Omega\times Y)}^2\\
	\leq C_m+C\left(\big\|\decongen{f_{\e_n}}{n}-\decongen{f_{\e_m}}{m}\big\|_{L^2(\Omega\times Y)}^2
	+\big\|\decongen{y_{\e_n}}{n}-\decongen{y_{\e_m}}{m}\big\|_{L^2(\Omega\times Y)}^2\right),
\end{multline}
where $\lim C_m=0$.
As $y_{\e}$ is a cauchy sequence (it converges strongly to the identity operator), it can also be estimated via a function $C_m$ converging to 0.  
The matrix valued function $B$, which is defined via \cref{eq:def:B}, is Lipschitz continuous with Lipschitz constant 2, i.e., 
\begin{equation}\label[in]{ineq:st2sc:4}
\big|B(\decongen{z_{\e_n}}{n})-B(\decongen{z_{\e_m}}{m})\big|\leq 2|\decongen{z_{\e_n}}{n}-\decongen{z_{\e_m}}{m}\big|.
\end{equation}
\end{subequations}
Adding \cref{ineq:st2sC_y,ineq:st2sc_z}, using the estimates given by \cref{ineq:st2sc:1,ineq:st2sc:2,ineq:st2sc:3} as well as Assumption (A3), and applying Gronwall's inequality, we infer
\begin{multline}
\big\|\e_n^{-1}\decongen{y_{\e_n}-\mathrm{Id}}{n}-\e_m^{-1}\decongen{y_{\e_m}-\mathrm{Id}}{m}\big\|^2_{L^2(\Omega\times Y)}
+\big\|\decongen{z_{\e_n}}{n}-\decongen{z_{\e_m}}{m}\big\|^2_{L^2(\Omega\times Y)}\\
\leq C_m+C\left(\left\|\decongen{v_{\e_n}}{n}-\decongen{v_{\e_m}}{m}\right\|^2+\big\|\e_n\decongen{\nabla v_{\e_n}}{n}-\e_m\decongen{\nabla v_{\e_m}}{m}\big\|^2\right)
\end{multline}
for all $n,m\in\N$ such that $n,m>N$ for sufficiently large $N\in\N$ (which is independent of $\e$ and $t$).
This implies
$$
\frac{1}{\e}\decon{y_\e-\mathrm{Id}}\to y-\mathrm{Id},\quad \decon{z_\e}\to z\quad\text{in}\ \ L^2(S\times\Omega\times Y)^3.
$$
Similarly, we also get (for more details, we refer to \cite[Lemma 6.21]{E18})
%
%
$$
\decon{Dy_\e}\to D_yy,\quad \e\decon{Dz_\e}\to D_yz\quad\text{in}\ \ L^2(S\times\Omega\times Y)^{3\times3}.
$$
\end{proof}

\begin{remark}
As a consequence of \Cref{lemma:trace_convergence}, this implies
\begin{equation*}
\frac{1}{\e}\decon{y_\e-\mathrm{Id}}\to y-\mathrm{Id},\quad \decon{z_\e}\to z\quad\text{in}\ \ L^2(S\times\Omega\times\Gamma)^3.
\end{equation*}
\end{remark}

%

\begin{lemma}
The following convergences hold:
$$
\frac{1}{\e}\decon{y_\e^{-1}-\mathrm{Id}}\to y^{-1}-\mathrm{Id},\quad \decon{q_\e}\to z(y^{-1}),\quad\e^{-1}\widetilde{\p}_\e\to\widetilde{\p},\quad\e\nabla q_\e\to\nabla_yq\quad\text{in} \ \ L^2(S\times\Omega\times Y).
$$
\end{lemma}
\begin{proof}
We recall that $y_\e^{-1}$ can be characterized by \cref{eq:Dtyinv}.
This leads us to
\begin{multline*}
\ddt\big\|\e_n^{-1}\decongen{y_{\e_n}^{-1}-\mathrm{Id}}{n}-\e_m^{-1}\decongen{y_{\e_m}^{-1}-\mathrm{Id}}{m}\big\|^2_{L^2(\Omega\times Y)}\\
\leq\int_{\Omega\times Y}\Big|Dy_{\e_n}(\decongen{y_{\e_n}^{-1}}{n})^{-1}\frac{z_{\e_n}(\decongen{y_{\e_n}^{-1}}{n})}{|z_{\e_n}(\decongen{y_{\e_n}^{-1}}{n})|}v_{\e_n}(y_{\e_n}(\decongen{y_{\e_n}^{-1}}{n}))\hspace{2cm}\\
\hspace{2cm}-Dy_{\e_m}(\decongen{y_{\e_m}^{-1}}{m})^{-1}\frac{z_{\e_m}(\decongen{y_{\e_m}^{-1}}{m})}{|z_{\e_m}(\decongen{y_{\e_m}^{-1}}{m})|}v_{\e_m}(y_{\e_m}(\decongen{y_{\e_m}^{-1}}{m}))\Big|\\
\cdot\big|\e_n^{-1}\decongen{y_{\e_n}^{-1}-\mathrm{Id}}{n}-\e_n^{-1}\decongen{y_{\e_n}^{-1}-\mathrm{Id}}{m}\big|\di{(x,y)}.
\end{multline*}
Taking into considerations the a-priori estimates available for the involved functions and the strong convergence results formulated in \Cref{lemma:twosC_yz}, as well as the estimates given in \Cref{lemma:decon_estimate}, it follows that $\e^{-1}\decon{y_\e^{-1}-\mathrm{Id}}$ is a Cauchy sequence.
Similarly, $\decon{q_\e}=\decon{z_\e(y_\e^{-1})}$ is a Cauchy sequence due to \Cref{lemma:decon_estimate}\,\eqref{it:lemma:decon_estimate:2}.
Since $\partial_t\widetilde{\p}_\e$ is governed by \cref{eq:motion_problem} and since $\nabla\widetilde{\p}_\e=q_\e$, we infer
\begin{equation*}
\ddt\big\|\e_n^{-1}\widetilde{\p}_{\e_n}-\e_m^{-1}\widetilde{\p}_{\e_m}\big\|^2_{L^2(\Omega\times Y)}
\leq\int_{\Omega\times Y}\big|\left|q_{\e_n}\right|v_{\e_n}
-\left|q_{\e_m}\right|v_{\e_m}\big|\di{(x,y)}
\end{equation*}
which shows that $\e^{-1}\widetilde{\p}_{\e}$ also converges strongly. 
Finally, as
$$
\e\nabla{q}_\e=\e D^2\widetilde{\p}_\e=\e\big(Dy_\e(y_\e^{-1})\big)^{-1}Dz_\e(y_\e^{-1}),
$$
we also get the strong convergence of $\e\decon{\nabla q_\e}$.
\end{proof}

Since the quantity $\e\|h_\e\|_\infty+\|\nabla_{\Gamma_\e}h_\e\|_\infty$ is bounded indepedently of the parameter $\e$, we can find a constant $C_h>0$ such that
$$
\frac{1}{\sqrt{\e}}\|h_\e\|_{L^2(S\times\Gamma_\e)}+\sqrt{\e}\|\nabla_{\Gamma_\e}h_\e\|_{L^2(S\times\Gamma_\e)^3}\leq C_h.
$$
As a result, we conclude the existence of a function $h\in L^2(S\times\Omega;H^1(\Gamma))$ such that, up to a subsequence, %
$$
\frac{1}{\e}h_\e\twosc h,\quad\nabla_{\Gamma_\e}h_\e\twosc\nabla_\Gamma h
$$
Furthermore, it is clear that $h\in L^\infty(S\times\Omega\times\Gamma)$ and that $\decon{h_\e}\in L^\infty(S\times\Omega\times\Gamma)$ is bounded independently of $\e$.
As a consequence, there is a function $\tilde{h}\in L^\infty(S\times\Omega\times\Gamma)$ such that $\decon{h_\e}\rightharpoonup \tilde{h}$ in $L^2(S\times\Omega\times Y)$.
In the following, we are concerned with the limit behavior of $h_\e$.

\begin{lemma}\label{lemma:trace_convergence2}
There is $h\in L^2(S\times\Omega;H^1_\#(\Gamma))$ such that $\e^{-1}\decon{h_\e}\to h$ and such that $\decon{\nabla_{\Gamma_\e}h_\e}\to\nabla_y h$ in $L^2(S\times\Omega\times\Gamma)$.
\end{lemma}
\begin{proof}
Let $\delta>0$ and $n,m\in\N$, $n>m$.
Using the representation of the height function $h_\e$ in terms of $F_{\e,\gamma}$ as given by \cref{eq:timeh_e}, we have
\begin{equation}\label{lemma:trace_convergence:eq3}
\partial_t h_\e(t,\gamma)=-\frac{\partial_tF_{\e,\gamma}(t, h_{\e}(t,\gamma))}{\partial_2F_{\e,\gamma}(t, h_{\e}(t,\gamma))}\quad (t\in[0,T_v],\, \gamma\in\Gamma_\e).
\end{equation}
Now, integrating over $\Omega\times\Gamma$ and testing with the difference $\e_n^{-1}\decongen{ h_{\e_n}}{n}-\e_m^{-1}\decongen{ h_{\e_n}}{m}$ leads to
\begin{multline*}
\ddt\left\|\e_n^{-1}\decongen{ h_{\e_n}}{n}-\e_m^{-1}\decongen{ h_{\e_n}}{m}\right\|_{L^2(\Omega\times\Gamma)}^2\\
\leq2\int_{\Omega\times \Gamma}\left|\e_n^{-1}\frac{\partial_t\decongen{F_{\e_n,\gamma}( h_{\e_n})}{n}}{\decongen{\partial_2F_{\e_n,\gamma}( h_{\e_n})}{n}}
-\e_m^{-1}\frac{\partial_t\decongen{F_{\e_m,\gamma}( h_{\e_m})}{m}}{\decongen{\partial_2F_{\e_m,\gamma}( h_{\e_m})}{m}}\right|\\
\big|\e_n^{-1}\decongen{ h_{\e_n}}{n}-\e_m^{-1}\decongen{ h_{\e_n}}{m}\big|\di{(x,\gamma)}.
\end{multline*}
Using that $\partial_t\widetilde{\p}_\e$ is governed by \cref{eq:motion_problem} and $q_\e=\nabla\widetilde{\p}_\e$ , we get
\begin{equation}\label{lemma:trace_convergence:eq4}
\e^{-1}\partial_t\decon{F_{\e,\gamma}( h_{\e})}=\left|q_{\e}(\decon{\eta_{\e}}))\right|v_{\e}(\decon{\eta_{\e}}).
\end{equation}
Applying \Cref{lemma:decon_estimate}\eqref{it:lemma:decon_estimate:3} to $q_\e$ and $v_\e$, respectively, and using the strong convergence of $\decon{v_\e}$, $\decon{\nabla v_\e}$, $\decon{q_\e}$, and $\e\decon{\nabla q_\e}$, we are led to
\begin{multline}\label[in]{lemma:trace_convergence:eq1}
\left\|\e_n^{-1}\partial_t\decongen{F_{\e_n,\gamma}( h_{\e_n})}{n}-\e_m^{-1}\partial_t\decongen{F_{\e_m,\gamma}( h_{\e_m})}{m}\right\|^2_{L^2(\Omega\times\Gamma)}\\
\leq C(m)+C\left(\big\|\decongen{h_{\e_n}}{n}-\decongen{h_{\e_m}}{m})\big\|^2_{L^2(\Omega\times\Gamma)}
+\big\|\e_n^{-1}\decongen{h_{\e_n}}{n}-\e_m^{-1}\decongen{h_{\e_n}}{n}\big\|^2_{L^2(\Omega\times\Gamma)}\right)
\end{multline}
where $\lim_{m\to\infty}C(m)=0$.
As a next step, we estimate the difference with respect to $\partial_2 F_{\e,\gamma}$.
In view of \cref{lemma_estimate_Fegamma_1}, we have
\begin{equation}\label{lemma:trace_convergence:eq5}
\decon{\partial_2F_{\e,\gamma}( h_{\e})}
=g'(\e^{-1}\widetilde{\p}_\e(\decon{\eta_\e}))q_\e(\decon{\eta_\e})\cdot n
\end{equation}
and, due to the strong convergence of $\e^{-1}\decon{\widetilde{\p}_\e}$, $\decon{q_\e}=\decon{\nabla\widetilde{\p}_\e}$, and $\e\decon{\nabla q_\e}$, we can infer (again applying \Cref{lemma:decon_estimate}\eqref{it:lemma:decon_estimate:3})
\begin{multline}\label[in]{lemma:trace_convergence:eq2}
\left\|\e_n^{-1}\decongen{\partial_2F_{\e_n,\gamma}( h_{\e_n})}{n}-\e_m^{-1}\decongen{\partial_2F_{\e_m,\gamma}( h_{\e_m})}{m}\right\|^2_{L^2(\Omega\times\Gamma)}\\
\leq C_m+C\left(\big\|\decongen{h_{\e_n}}{n}-\decongen{h_{\e_m}}{m})\big\|^2_{L^2(\Omega\times\Gamma)}+\big\|\e_n^{-1}\decongen{h_{\e_n}}{n}-\e_m^{-1}\decongen{h_{\e_n}}{n}\big\|^2_{L^2(\Omega\times\Gamma)}\right)
\end{multline}
where $\lim_{m\to\infty}C_m\to0$.
Combining the estimates given by \cref{lemma:trace_convergence:eq1,lemma:trace_convergence:eq2} and applying Gronewall's inequality, it is then easy to see that $\e^{-1}\decon{h_{\e}}$ is, in fact, Cauchy.
%
%

Using the representation of $h_\e$ given in \cref{eq:surface_gradient_h}, we have
\begin{equation*}
\decon{\nabla_{\Gamma_\e} h_\e}=\left(\mathds{I}_3-\e^{-1}\decon{h_\e}L_{\Gamma}\right)\left(n-\frac{1}{n_{\Gamma_\e}(\decon{\eta_\e})\cdot n}n_{\Gamma_\e}(\decon{\eta_\e})\right).
\end{equation*}
Consequently, since $n_{\Gamma_\e}(\eta_\e)\cdot n_\e>\nicefrac{1}{2}$ and $|\e^{-1}h_\e|\leq\nicefrac{a}{10}$ in $[0,T_v]\times\Gamma_\e$, we are led to
\begin{multline*}
\left\|\decongen{\nabla_{\Gamma_{\e_n}} h_{\e_n}}{n}-\decongen{\nabla_{\Gamma_{\e_m}} h_{\e_m}}{m}\right\|_{L^2(S\times\Gamma_\e)}\\
\leq\frac{3}{2a}\left\|\e_n^{-1}\decon{h_{\e_n}}-\e_m^{-1}\decon{h_{\e_m}}\right\|_{L^2(S\times\Gamma_\e)}
+6\left\|\frac{n_{\e_n}(\decongen{\eta_{\e_n}}{n})}{n_{\e_n}(\decongen{\eta_{\e_n}}{n})\cdot n}-\frac{n_{\e_m}(\decongen{\eta_{\e_m}}{m})}{n_{\e_m}(\decongen{\eta_{\e_m}}{m})\cdot n}\right\|_{L^2(S\times\Gamma_\e)}
\end{multline*}
Now, due to $n_{\Gamma_\e}(\eta_\e)=\frac{\nabla\widetilde{\p}_\e(\eta_\e)}{|\nabla\widetilde{\p}_\e(\eta_\e)|}=\frac{q_\e(\eta_\e)}{|q_\e(\eta_\e)|}$, we further estimate
\begin{align*}
\bigg\|\frac{n_{\e_n}(\decongen{\eta_{\e_n}}{n})}{n_{\e_n}(\decongen{\eta_{\e_n}}{n})\cdot n}-\frac{n_{\e_m}(\decongen{\eta_{\e_m}}{m})}{n_{\e_m}(\decongen{\eta_{\e_m}}{m})\cdot n}&\bigg\|_{L^2(S\times\Gamma_\e)}\\
&\leq6\left\|\frac{q_{\e_n}(\decongen{\eta_{\e_n}}{n})}{|q_{\e_n}(\decongen{\eta_{\e_n}}{n})|}-\frac{q_{\e_m}(\decongen{\eta_{\e_m}}{m})}{|q_{\e_m}(\decongen{\eta_{\e_m}}{m})|}\right\|_{L^2(S\times\Gamma_\e)}\\
&\leq36\left\|q_{\e_n}(\decongen{\eta_{\e_n}}{n})-q_{\e_m}(\decongen{\eta_{\e_m}}{m})\right\|_{L^2(S\times\Gamma_\e)}.
\end{align*}
As both $\decon{q_\e}$ and $\e\decon{\nabla q_\e}$ converge, we can apply \Cref{lemma:decon_estimate}\eqref{it:lemma:decon_estimate:3} and conclude
\begin{multline*}
\left\|\decongen{\nabla_{\Gamma_{\e_n}} h_{\e_n}}{n}-\decongen{\nabla_{\Gamma_{\e_m}} h_{\e_m}}{m}\right\|_{L^2(S\times\Gamma_\e)}\\
\leq C_m +\frac{3}{2a}\left\|\e_n^{-1}\decon{h_{\e_n}}-\e_m^{-1}\decon{h_{\e_m}}\right\|_{L^2(S\times\Gamma_\e)}\\
+C\left(\left\|\decon{h_{\e_n}}-\decon{h_{\e_m}}\right\|_{L^2(S\times\Gamma_\e)}
+\left\|\e_n^{-1}\decon{h_{\e_n}}-\e_m^{-1}\decon{h_{\e_m}}\right\|_{L^2(S\times\Gamma_\e)}\right),
\end{multline*}
where, again, $\lim_{m\to\infty} C_m=0$.

\end{proof}
We introduce $\psi_\e=s_\e-\mathrm{Id}$ which implies (see \cref{def_se}) $D\psi_\e=D s_\e$.

\begin{lemma}\label{mi:lem:psilimit}
There is $\psi\in L^2(S\times\Omega;H^{1}_\#(Y))$ such that $\e^{-1}\decon{\psi_\e}\to\psi$ and such that $\decon{\nabla\psi_\e}\to\nabla_y\psi$ in $L^2(S\times\Omega\times Y)$.
\end{lemma}
\begin{proof}
Let $n,m\in\N$ such that $m>n$ and set $\mu_\e(t,x)=h_\e(t,P_{\Gamma_\e}(x))$ as well as $\mu(t,x,y)=h(t,x,P_\Gamma(y))$.
We calculate
\begin{align*}
\e_n^{-1}\decongen{\psi_{\e_n}}{n}-\e_m^{-1}\decongen{\psi_{\e_m}}{m}
	&=\Big(\e_n^{-1}\decongen{\mu_{\e_n}}{n}-\e_m^{-1}\decongen{\mu_{\e_m}}{m}\Big)\chi\left(a^{-1}d_{\Gamma}\right)n(P_\Gamma).
\end{align*}
As a consequence,
\begin{multline*}
\int_{\Omega\times U_\Gamma}\left|\e_n^{-1}\decongen{\psi_{\e_n}}{n}-\e_m^{-1}\decongen{\psi_{\e_m}}{m}\right|^2\di{(x,y)}\\
\leq\int_{\Omega\times U_\Gamma}\left|\e_n^{-1}\decongen{\mu_{\e_n}}{n}-\mu\right|^2
+\left|\e_m^{-1}\decongen{\mu_{\e_m}}{m}-\mu\right|^2\di{(x,y)}.
\end{multline*}
Now, for fixed $x\in\Omega$, $\decon{\mu_\e}$ and $\mu$ are constant in the $y$ variable in the direction of the normal vector.
As a consequence,
\begin{equation*}
\int_{\Omega\times U_\Gamma}\left|\e^{-1}\decon{\mu_{\e}}-\mu\right|^2\di{(x,y)}=2a\int_{\Omega\times\Gamma}\left|\e_n^{-1}\decongen{h_{\e_n}}{n}-h\right|^2\di{(x,y)}.
\end{equation*}
The unfolded deformation gradient is given via (we refer to \cite[Section 2]{PSZ13})
\begin{multline*}
\decon{\nabla\psi_\e}=\left(\decon{\nabla\mu_\e}\right)^Tn(P_{\Gamma})\chi\left(a^{-1}d_{\Gamma}\right)\\
+\e^{-1}\decon{\mu_\e}\Big(L_\Gamma(P_\Gamma)\left(\mathds{I}-d_\Gamma L_\Gamma(P_\Gamma)\right)^{-1}(\mathds{I}-n(P_\Gamma)\otimes n(P_\Gamma))\chi\left(a^{-1}d_{\Gamma}\right)\\
+\chi'\left(a^{-1}d_{\Gamma}\right)n(P_\Gamma)\otimes n(P_\Gamma)\Big).
\end{multline*}
which leads us to
\begin{multline*}
\int_{\Omega\times U_\Gamma}\left|\decongen{\nabla\psi_{\e_n}}{n}-\decongen{\nabla\psi_{\e_m}}{m}\right|^2\di{(x,y)}\\
\leq C\int_{\Omega\times U_\Gamma}\left|\e_n^{-1}\decongen{\mu_{\e_n}}{n}-\e_m^{-1}\decongen{\mu_{\e_m}}{m}\right|^2
+\left|\decongen{\nabla\mu_{\e_n}}{n}-\decongen{\nabla\mu_{\e_m}}{m}\right|^2\di{(x,y)},
\end{multline*}
where $C>0$ is independent of $\e$.
Since
$$
\nabla\mu_{\e}(t,x)=\left(DP_{\Gamma_\e}(x)\right)^T\nabla_{\Gamma_\e}h_\e(t,P_{\Gamma_\e}(x))
$$
and
\begin{multline*}
\int_{\Omega\times U_\Gamma}\left|\decon{\nabla_{\Gamma_\e}h_\e(P_{\Gamma_\e})}(x,y)-\nabla_yh(t,x,P_{\Gamma}(y))\right|^2\di{(x,y)}\\
=2a\int_{\Omega\times\Gamma}\left|\decongen{\nabla_{\Gamma_\e}h_{\e_n}}{n}-\nabla_yh\right|^2\di{(x,y)},
\end{multline*}
we can conclude $\decon{\nabla\psi_{\e}}\to\nabla_y\psi$.
\end{proof}
\begin{remark}
Looking at the Definition of $s_\e$ given via \cref{def_se}, it is then clear that both $Ds_\e$ and $\det Ds_\e$ are also strongly two-scale convergent.
Due to the uniform boundedness in $L^\infty$, it is also clear that the limit functions are essentially bounded.
\end{remark}

\bibliography{literature}{}
\bibliographystyle{plain}

\end{document}